\documentclass[10pt, reqno]{amsart}
\usepackage{bbm}
\usepackage{cases}
\usepackage{txfonts}
\usepackage{amsfonts}
\usepackage{mathrsfs}
\usepackage{amssymb}
\usepackage{amsmath}
\usepackage{epic}

\newtheorem{proposition}{Proposition}[section]
\newtheorem{lemma}[proposition]{Lemma}
\newtheorem{corollary}[proposition]{Corollary}
\newtheorem{theorem}[proposition]{Theorem}

\theoremstyle{definition}

\theoremstyle{remark}
\newtheorem{remark}[proposition]{Remark}

\newtheorem{notation}[proposition]{Notation}

\newcommand{\thlabel}[1]{\label{th:#1}}
\newcommand{\thref}[1]{Theorem~\ref{th:#1}}
\newcommand{\selabel}[1]{\label{se:#1}}
\newcommand{\seref}[1]{Section~\ref{se:#1}}
\newcommand{\lelabel}[1]{\label{le:#1}}
\newcommand{\leref}[1]{Lemma~\ref{le:#1}}
\newcommand{\prlabel}[1]{\label{pr:#1}}
\newcommand{\prref}[1]{Proposition~\ref{pr:#1}}
\newcommand{\colabel}[1]{\label{co:#1}}
\newcommand{\coref}[1]{Corollary~\ref{co:#1}}

\newcommand{\nolabel}[1]{\label{no:#1}}
\newcommand{\noref}[1]{Notation~\ref{no:#1}}
\def\a{\alpha}
\def\b{\beta}

\def\D{\Delta}
\def\e{\epsilon}

\def\ep{\varepsilon}

\def\g{\gamma}

\def\l{\lambda}

\def\ol{\overline}
\def\op{\oplus}
\def\ot{\otimes}

\def\oo{\infty}

\def\ra{\rightarrow}
\def\s{\sigma}

\def\ti{\times}

\def\<{\leqslant}
\def\>{\geqslant}

\date{}
\begin{document}
\title[Green ring of the category of weight modules]
{Green ring of the category of weight modules over the Hopf-Ore extensions of group algebras}

\author{Hua Sun}
\address{College of Mathematical Science, Yangzhou University,
Yangzhou 225002, China}
\email{997749901@qq.com}

\author{Hui-Xiang Chen}
\address{College of Mathematical Science, Yangzhou University,
Yangzhou 225002, China}
\email{hxchen@yzu.edu.cn}


\begin{abstract}
In this paper, we continue our study of the tensor product structure of category $\mathcal W$ of weight modules
over the Hopf-Ore extensions $kG(\chi^{-1}, a, 0)$ of group algebras $kG$,
where $k$ is an algebraically closed field of characteristic zero.
We first describe the tensor product decomposition rules for all indecomposable weight modules
under the assumption that the orders of $\chi$ and $\chi(a)$ are different. Then we describe
the Green ring $r(\mathcal W)$ of the tensor category $\mathcal W$. It is shown that $r(\mathcal W)$ is isomorphic to
the polynomial algebra over the group ring $\mathbb{Z}\hat{G}$ in one variable
when $|\chi(a)|=|\chi|=\infty$, and that $r(\mathcal W)$ is isomorphic to the quotient ring
of the polynomial algebra over the group ring $\mathbb{Z}\hat{G}$ in two variables modulo a principle ideal
when $|\chi(a)|<|\chi|=\infty$. When $|\chi(a)|\<|\chi|<\infty$,
$r(\mathcal W)$ is isomorphic to the quotient ring of a skew group ring $\mathbb{Z}[X]\sharp\hat{G}$
modulo some ideal, where $\mathbb{Z}[X]$ is a polynomial algebra over $\mathbb{Z}$ in infinitely many variables.
\end{abstract}
\maketitle

\section{Introduction and Preliminaries}\selabel{1}
In \cite{SunChen}, we studied the tensor product structure of the category of finite dimensional weight modules
over the Hopf-Ore extensions $kG(\chi^{-1}, a, 0)$ of group algebras $kG$.
Hopf-Ore extensions were introduced and studied by Panov in \cite {Pa}. Krop, Radford and Scherotzke
studied finite dimensional rank one pointed Hopf algebras over an algebraically closed field in \cite{KR} and \cite{Sc},
respectively. They showed that a finite dimensional rank one pointed Hopf algebra over an algebraically closed
field is isomorphic to a quotient of a Hopf-Ore extension of a group algebra.
This result was generalized to the case of such Hopf algebras over an arbitrary field by Wang, You and Chen in \cite{WangYouChen}.
They also studied the representations of the Hopf-Ore extensions of
group algebras and rank one pointed Hopf algebras over an arbitrary field, and classified the finite dimensional
indecomposable weight modules over such Hopf algebras \cite{WangYouChen}.
In \cite{ChVOZh} and \cite{LiZhang}, the authors studied respectively
the Green rings the Taft algebras and the generalized Taft algebras based on the Cibil's work \cite{Cib}.
Taft algebras and generalized Taft algebras are finite dimensional rank one pointed Hopf algebras.
Wang, Li and Zhang \cite{WangLiZhang2014, WangLiZhang2016} studied the Green rings of finite dimensional
rank one pointed Hopf algebras over an algebraically closed field of characteristic zero.
In \cite{SunChen}, we described the decomposition rules
for tensor product of finite dimensional indecomposable weight modules over the Hopf-Ore extensions
$kG(\chi^{-1}, a, 0)$ of group algebras $kG$ in case $|\chi|=|\chi(a)|$, where the ground field $k$
is an algebraically closed field of characteristic zero.

In this paper, we continue the study of the tensor product structure of the category $\mathcal W$
of finite dimensional weight modules over the Hopf-Ore extensions $kG(\chi^{-1}, a, 0)$ of group algebras $kG$
over an algebraically closed field $k$ of characteristic zero,
but concentrate on the decomposition rules for tensor product modules in case $|\chi|\neq|\chi(a)|$,
and the Green rings of $r(\mathcal W)$ in the both cases: $|\chi|=|\chi(a)|$ and $|\chi|\neq|\chi(a)|$.
The finite dimensional indecomposable weight modules can be clarified
into two types: nilpotent type and non-nilpotent type. For the tensor products of two nilpotent type modules, or a nilpotent type
module with a non-nilpotent one, the decomposition rules for $|\chi|\neq|\chi(a)|$ are similar to those in the case of
$|\chi|=|\chi(a)|$. However, as we shall see in the case of the tensor products of two non-nilpotent type modules,
the decomposition rules for $|\chi|\neq|\chi(a)|$ are much more complicated than those for $|\chi|=|\chi(a)|$.
The paper is organized as follows. In this section, we recall the Hopf algebra structure of
$H=kG(\chi^{-1}, a, 0)$ and some notations.
In \seref{2}, we recall the classification of finite dimensional indecomposable
weight modules over $H$. In \seref{3}, we consider the tensor products of finite dimensional indecomposable weight modules
over $H$ in case $|\chi|\neq|\chi(a)|$. We give the decomposition rules for all such tensor product modules.
In \seref{4}, we investigate the Green ring $r(\mathcal W)$ of the tensor category $\mathcal W$.
We describe the Green ring $r(\mathcal W)$ in terms of polynomial ring and skew group ring. $r(\mathcal W)$ is a commutative ring
when $|\chi|=\infty$, but $r(\mathcal W)$ is not when $|\chi|<\infty$.
It is shown that $r(\mathcal W)$ is isomorphic to the polynomial algebra over the group ring $\mathbb{Z}\hat{G}$ in one variable
in case $|\chi(a)|=|\chi|=\infty$, and that $r(\mathcal W)$ is isomorphic to the quotient ring
of the polynomial algebra over $\mathbb{Z}\hat{G}$ in two variables modulo a principle ideal
in case $|\chi(a)|<|\chi|=\infty$. For the case of $|\chi(a)|\<|\chi|<\infty$,
we show that $r(\mathcal W)$ is isomorphic to the quotient ring of a skew group ring $\mathbb{Z}[X]\sharp\hat{G}$
modulo some ideal, where $\mathbb{Z}[X]$ is a polynomial algebra over $\mathbb{Z}$ in infinitely many variables.

Throughout, we work over a field $k$. Unless otherwise stated, all algebras, coalgebras,
Hopf algebras and modules are vector spaces over $k$. All linear maps mean $k$-linear maps,
dim and $\otimes$ mean dim$_k$ and $\otimes_k$, respectively.
Our references for basic concepts and notations about representation theory and Hopf algebras
are \cite{ARS, Kas, Mo}.
In particular, for a Hopf algebra, we will use $\ep$, $\D$ and $S$ to denote the counit,
comultiplication and antipode, respectively. Let $k^{\ti}=k\setminus\{0\}$.
For a group $G$, let $\hat{G}$ denote the group of characters of $G$ over $k$,
and let $Z(G)$ denote the center of $G$. For a Hopf algebra $H$,
an $H$-module means a left $H$-module. Let $\mathbb Z$ denote all integers.

For a Hopf algebra $H$, the category mod$H$ of finite dimensional $H$-module is a tensor (or monoidal) category \cite{Kas, Mo}.
The Green ring (or representation ring) $r(H)$ of $H$ is defined to be the abelian group generated by the
isomorphism classes $[V]$ of $V$ in mod$H$
modulo the relations $[U\oplus V]=[U]+[V]$, $U, V\in{\rm mod}H$. The multiplication of $r(H)$
is determined by $[U][V]=[U\ot V]$, the tensor product of $H$-modules. Then $r(H)$ is an associative ring
with identity. Notice that $r(H)$ is a free abelian group with a $\mathbb Z$-basis
$\{[V]|V\in{\rm ind}(H)\}$, where ${\rm ind}(H)$ denotes the category
of indecomposable objects in mod$H$.

Let $G$ be a group and $a\in Z(G)$. Let $\chi\in\hat{G}$ with $\chi(a)\neq1$.
The Hopf-Ore extension $H=kG(\chi^{-1}, a, 0)$ of the group
algebra $kG$ can be described as follows. $H$ is generated, as an algebra, by $G$ and $x$ subject to
the relations $xg=\chi^{-1}(g)gx$ for all $g\in G$. The coalgebra structure and antipode are given
by
$$\begin{array}{lll}
\D(x)=x\ot a+1\ot x,& \ep(x)=0,& S(x)=-xa^{-1},\\
\D(g)=g\ot g,& \ep(g)=1,& S(g)=g^{-1},\\
\end{array}$$
where $g\in G$.
$H$ has a $k$-basis $\{gx^i \mid g\in G,i\>0\}$.

For any positive integer $n$ and a scale $\a\in k$, let $J_n(\a)$ be the $n\times n$-matrix
$$\left(\begin{array}{cccc}
\a&&&\\
1&\a&&\\
&\ddots&\ddots&\\
&&1&\a\\
\end{array}
\right)$$
over $k$, the Jordan block matrix.

Let $0\not=q\in k$. For any nonnegative integer $n$, define $(n)_q$ by $(0)_q=0$ and $(n)_q=1+q+\cdots +q^{n-1}$ for $n>0$.
Observe that $(n)_q=n$ when $q=1$, and
$$
(n)_q=\frac{q^n-1}{q-1}
$$
when $q\not= 1$.
Define the $q$-factorial of $n$ by
$(0)!_q=1$ and
$(n)!_q=(n)_q(n-1)_q\cdots (1)_q$ for $n>0$.
Note that $(n)!_q=n!$ when $q=1$, and
$$
(n)!_q=
\frac{(q^n-1)(q^{n-1}-1)\cdots (q-1)}{(q-1)^n}
$$
when $n>0$ and $q\not= 1$.
The  $q$-binomial coefficients
$\binom{n}{i}_q$
are defined inductively as follows for $0\leqslant i\leqslant n$:
$$
\binom{n}{0}_q=1=\binom{n}{n}_q
\quad\quad
\mbox{ for } n\geqslant 0,$$
$$
\binom{n}{i}_q=q^i\binom{n-1}{i}_q+\binom{n-1}{i-1}_q
\quad \quad
\mbox{ for } 0< i< n.$$
It is well-known that $\binom{n}{i}_q$
is a polynomial in $q$ with integer coefficients and with value at $q=1$
equal to the usual binomial coefficient $\binom{n}{i}$, and that
$$
\binom{n}{i}_q
=\frac{(n)!_q}{(i)!_q(n-i)!_q}
$$
when
$(n-1)!_q\not = 0$ and $0<i<n$
(see \cite[Page 74]{Kas}).

\section{Indecomposable weight modules }\selabel{2}

In this section, we recall the finite dimensional indecomposable weight $H$-modules
(see \cite{WangYouChen, SunChen}. We still assume that $\chi^{-1}(a)\neq 1$,
and use the notations of last section. Throughout the following, let
$H=kG(\chi^{-1}, a, 0)$ be the Hopf algebra given as in the last section.
For an $H$-module $M$ and $m\in M$, let $\langle m\rangle$ denote the submodule of $M$ generated by $m$.

Let $M$ be an $H$-module. For any $\l \in \hat{G}$,
let $M_{(\l)}=\{v\in M\mid g\cdot v=\l (g)v,\, g\in G\}$.
Each nonzero element in $M_{(\l)}$ is called a {\it weight vector of weight $\l$} in $M$.
One can check that $\op _{\l \in \hat{G}} M_{(\l)}$
is a submodule of $M$.
Let $\Pi(M)=\{\l \in \hat{G}\mid M_{(\l)}\neq 0\}$, which is
called the {\it weight space} of $M$.
$M$ is said to be a {\it weight module} if $M=\op _{\l\in \Pi(M)}
M_{(\l)}$. Let $\mathcal W$ be the full subcategory of mod$H$ consisting of all finite dimensional weight modules.
Then $\mathcal W$ is a tensor subcategory of mod$H$.

For any $\l \in \hat{G}$ and $t\in\mathbb{Z}$ with $t\>1$, let $V_t(\l)$ be a $t$-dimensional $k$-space
with a basis $\{m_0, m_1, \cdots, m_{t-1}\}$.
Then one can easily check $V_t(\l)$ is an $H$-weight modules with the action determined by
$$ g\cdot m_i=\chi^i(g)\l(g)m_i,\  x\cdot m_i=m_{i+1}, 0\<i\<t-2, \ x\cdot m_{t-1}=0.$$

When $|\chi|={\ol s}<\infty$, for any $\l\in \hat{G}$, $\b \in k$, $t\in\mathbb{Z}$ with $t\>1$, let $V_t(\l,\b)$
be a vector space of dimension $t{\ol s}$ with a basis $\{m_0, m_1, \cdots, m_{t{\ol s}-1}\}$ over $k$.
Assume that $(y-\b)^t=y^t-\sum_{j=0}^{t-1}\a_jy^j$, where $\a_0, \a_1, \cdots, \a_{t-1}\in k$.
Then one can easily check that $V_t(\l,\b)$ is an $H$-module with the action determined by
$$ g\cdot m_i=\chi^i(g)\l(g)m_i,\ \ x\cdot m_i=\left\{\begin{array}{ll}
m_{i+1},&0\<i\<t{\ol s}-2\\
\sum_{j=0}^{t-1}\a_j m_{j{\ol s}},&i=t{\ol s}-1\\
\end{array}\right.,$$
where $0\<i\<t{\ol s}-1$, $g\in G$. Obviously, $V_t(\l, \b)$ is a weight module
and $V_t(\l, 0)=V_{t{\ol s}}(\l)$.

Let $\langle\chi\rangle$ denote the subgroup of
$\hat{G}$ generated by $\chi$, and $[\l]$ denote the image of $\l$ under the canonical
epimorphism $\hat{G}\ra\hat{G}/\langle\chi\rangle$.
Let $\e$ denote the identity element of the group $\hat{G}$, i.e., $\e(g)=1$, $\forall g\in G$.
The finite dimensional indecomposable weight $H$-modules are classified in \cite{WangYouChen}.
We state the classification (in case $k$ is an algebraically closed field) as follows (see \cite{SunChen}).

\begin{proposition}\prlabel{2.1}
Assume that $|\chi|=\infty$. Then $\{V_t(\l)|t\>1, \l\in\hat{G}\}$
is a complete set of finite dimensional indecomposable weight $H$-modules
up to isomorphism.
\end{proposition}

\begin{proposition}\prlabel{2.2}
Assume that $|\chi|={\ol s}<\oo$ and $k$ is algebraically closed. Then
$$\left\{V_t(\l), V_t(\s, \b)\left| \l\in\hat{G}, [\s]\in\hat{G}/\langle\chi\rangle, t\>1, \b\in k^{\times}\right.\right\}$$
is a complete set of finite dimensional indecomposable weight $H$-modules up to isomorphism.
\end{proposition}

\begin{remark}
For any $t\>1$, $\l\in\hat{G}$ and $\b\in k^{\times}$, the linear endomorphism of $V_t(\l)$
induced by the action of $x$ is nilpotent. However, the linear endomorphism of $V_t(\l, \b)$
induced by the action of $x$ is invertible. In the following, $V_t(\l)$ is called a module of
{\it nilpotent type}, and  $V_t(\l, \b)$ is called a module of {\it non-nilpotent type}.
\end{remark}

In what follows, assume that $k$ is an algebraically closed field of characteristic zero.

\section{Decomposition rules for tensor products of weight modules}\selabel{3}

In this section, we will investigate the tensor products of finite dimensional indecomposable weight modules over
$H$, and decompose such tensor products into the direct sum of indecomposable modules
in case $|\chi|\neq |\chi(a)|$, where $|\chi|$ and $|\chi(a)|$ denote the orders of
$\chi$ and $\chi(a)$, respectively.

{\bf Convention}: If $\oplus_{l\leqslant i\leqslant m}M_i$ is a term in a decomposition of a module,
then it disappears when $l>m$. For a module $M$ and a nonnegative integer $r$, let $rM$ denote
the direct sum of $r$ copies of $M$. In particular, $rM=0$ when $r=0$.
Let $V_0(\l)=V_0(\l,\b)=0$ for any $\l\in\hat{G}$ and $\b\in k$.

\subsection{Tensor products $V_n(\l)\ot V_t(\s)$}\selabel{3.1}
In this subsection, we consider the tensor products ot two modules of nilpotent type.
Throughout this subsection, assume that $|\chi(a)|=s<|\chi|\<\infty$.
Let  $q=\chi^{-1}(a)$. Then $q$ is a primitive $s^{th}$ root of unity.
In this case, \cite[Lemmas 3.8, 3.9, 3.10 and 3.11]{SunChen} still hold with the same proofs.

\begin{lemma}\lelabel{3.5}
Let $\l, \s\in\hat{G}$ and $t\in\mathbb Z$ with $t\>1$.\\
{\rm (1)} If $s|t$, then $V_{s+1}(\l)\ot V_t(\s)\cong V_t(\s)\ot V_{s+1}(\l)
\cong V_{t-s}(\chi^{s}\l\s)\oplus V_{t+s}(\l\s)\oplus(\oplus_{i=1}^{s-1}V_t(\chi^i\l\s))$.\\
$(2)$ Assume that $s\nmid t$ and let $t=rs+l$ with $1\<l\<s-1$.\\
\mbox{\hspace{0.4cm}\rm (a)} If $r=0$, then
$V_{s+1}(\l)\ot V_t(\s)\cong V_t(\s)\ot V_{s+1}(\l)
\cong V_{s+l}(\l\s)\oplus(\oplus_{1\<i\<l-1}V_{s}(\chi^i\l\s)).$
\mbox{\hspace{0.4cm}\rm (b)} If $r\>1$, then
$$\begin{array}{rl}
&V_{s+1}(\l)\ot V_t(\s)\cong V_t(\s)\ot V_{s+1}(\l)\\
\cong&V_{t+s}(\l\s)\oplus(\oplus_{1\<i\<l-1}V_{(r+1)s}(\chi^i\l\s))\oplus V_{t+s-2l}(\chi^l\l\s)\\
&\oplus(\oplus_{l+1\<i\<s-1}V_{rs}(\chi^i\l\s))\oplus V_{t-s}(\chi^{s}\l\s).\\
\end{array}$$
\end{lemma}

\begin{proof}
It is similar to \cite[Lemma 3.12]{SunChen}.
\end{proof}

\begin{lemma}\lelabel{3.6}
Let $\l, \s\in\hat{G}$ and $n, t\in\mathbb Z$ with $n, t\>1$. Assume $s|t$ and let
$t=rs$ and $n=r's+l$ with $0\<l\<s-1$. Then
$$\begin{array}{rl}
&V_n(\l)\ot V_t(\s)\cong V_t(\s)\ot V_n(\l)\\
\cong&(\oplus_{i=0}^{{\rm min}\{r',r-1\}}\oplus_{0\<j\<l-1}V_{(r+r'-2i)s}(\chi^{j+is}\l\s))
\oplus(\oplus_{0\<i\<{\rm min}\{r,r'\}-1}\oplus_{j=l}^{s-1}V_{(r+r'-1-2i)s}(\chi^{j+is}\l\s)).\\
\end{array}$$
\end{lemma}

\begin{proof}
It is similar to \cite[Lemma 3.13]{SunChen}.
\end{proof}

\begin{lemma}\lelabel{3.7}
Let $\l, \s\in\hat{G}$ and $n, r\in\mathbb Z$ with $n\>1$ and $r\>0$.
Assume $s\nmid n$ and let $n=r's+l$ with $1\<l\<s-1$ and $r'\>0$. Then we have
$$\begin{array}{rl}
&V_n(\l)\ot V_{rs+1}(\s)\cong V_{rs+1}(\s)\ot V_n(\l)\\
\cong&(\oplus_{i=0}^{{\rm min}\{r', r\}}V_{(r+r'-2i)s+l}(\chi^{is}\l\s))
\oplus(\oplus_{0\<i\<{\rm min}\{r', r-1\}}\oplus_{1\<j\<l-1}V_{(r+r'-2i)s}(\chi^{j+is}\l\s))\\
&\oplus(\oplus_{0\<i\<{\rm min}\{r', r\}-1}V_{(r+r'-2i)s-l}(\chi^{l+is}\l\s))\\
&\oplus(\oplus_{0\<i\<{\rm min}\{r', r\}-1}\oplus_{l+1\<j\<s-1}V_{(r+r'-1-2i)s}(\chi^{j+is}\l\s)).\\
\end{array}$$
\end{lemma}

\begin{proof}
It is similar \cite[Lemma 3.14]{SunChen}.
\end{proof}

\begin{proposition}\prlabel{3.8}
Let $\l, \s\in\hat{G}$ and $n, t\in\mathbb Z$ with $n\>t\>1$.
Assume that $n=r's+l'$ and $t=rs+l$ with $0\<l', l\<s-1$.\\
{\rm (1)} Suppose that $l+l'\<s$. If $l\<l'$ then
$$\begin{array}{rl}
&V_n(\l)\ot V_t(\s)\cong V_t(\s)\ot V_n(\l)\\
\cong&(\oplus_{i=0}^{r}\oplus_{0\<j\<l-1}V_{n+t-1-2is-2j}(\chi^{j+is}\l\s))
\oplus(\oplus_{0\<i\<r-1}\oplus_{l\<j\<l'-1}V_{(r+r'-2i)s}(\chi^{j+is}\l\s))\\
&\oplus(\oplus_{0\<i\<r-1}\oplus_{l'\<j\<l+l'-1}V_{n+t-1-2is-2j}(\chi^{j+is}\l\s))
\oplus(\oplus_{0\<i\<r-1}\oplus_{l+l'\<j\<s-1}V_{(r+r'-1-2i)s}(\chi^{j+is}\l\s)),\\
\end{array}$$
and if $l\>l'$ then
$$\begin{array}{rl}
&V_n(\l)\ot V_t(\s)\cong V_t(\s)\ot V_n(\l)\\
\cong&(\oplus_{i=0}^{r}\oplus_{0\<j\<l'-1}V_{n+t-1-2is-2j}(\chi^{j+is}\l\s))
\oplus(\oplus_{i=0}^{r}\oplus_{l'\<j\<l-1}V_{(r+r'-2i)s}(\chi^{j+is}\l\s))\\
&\oplus(\oplus_{0\<i\<r-1}\oplus_{l\<j\<l+l'-1}V_{n+t-1-2is-2j}(\chi^{j+is}\l\s))
\oplus(\oplus_{0\<i\<r-1}\oplus_{l+l'\<j\<s-1}V_{(r+r'-1-2i)s}(\chi^{j+is}\l\s)).\\
\end{array}$$
{\rm (2)} Suppose that $l+l'\>s+1$ and let $m=l+l'-s-1$. If $l\<l'$ then
$$\begin{array}{rl}
&V_n(\l)\ot V_t(\s)\cong V_t(\s)\ot V_n(\l)\\
\cong&(\oplus_{i=0}^{r}\oplus_{j=0}^mV_{(r+r'+1-2i)s}(\chi^{j+is}\l\s))
\oplus(\oplus_{i=0}^{r}\oplus_{j=m+1}^{l-1}V_{n+t-1-2is-2j}(\chi^{j+is}\l\s))\\
&\oplus(\oplus_{0\<i\<r-1}\oplus_{l\<j\<l'-1}V_{(r+r'-2i)s}(\chi^{j+is}\l\s))
\oplus(\oplus_{0\<i\<r-1}\oplus_{j=l'}^{s-1}V_{n+t-1-2is-2j}(\chi^{j+is}\l\s)),\\
\end{array}$$
and if $l\>l'$ then
$$\begin{array}{rl}
&V_n(\l)\ot V_t(\s)\cong V_t(\s)\ot V_n(\l)\\
\cong&(\oplus_{i=0}^{r}\oplus_{j=0}^{m}V_{(r+r'+1-2i)s}(\chi^{j+is}\l\s))
\oplus(\oplus_{i=0}^{r}\oplus_{j=m+1}^{l'-1}V_{n+t-1-2is-2j}(\chi^{j+is}\l\s))\\
&\oplus(\oplus_{i=0}^{r}\oplus_{l'\<j\<l-1}V_{(r+r'-2i)s}(\chi^{j+is}\l\s))
\oplus(\oplus_{0\<i\<r-1}\oplus_{j=l}^{s-1}V_{n+t-1-2is-2j}(\chi^{j+is}\l\s)).\\
\end{array}$$
\end{proposition}

\begin{proof}
It is similar to \cite[Theorem 3.15]{SunChen}.
\end{proof}

\subsection{Tensor products $V_n(\l)\ot V_t(\s, \b)$ and $V_t(\s, \b)\ot V_n(\l)$}\selabel{3.2}
In this subsection, we consider the tensor products of a module of nilpotent type with one of non-nilpotent type.
Throughout this and the next subsections, assume that $1<s=|\chi(a)|<|\chi|=\ol{s}<\infty$.
Let $q=\chi^{-1}(a)$ and $\ol{s}=ss'$. The $q$ is a primitive $s^{th}$ root of unity and $s'>1$.

\begin{lemma}\lelabel{3.10}
Let $\s,\l \in\hat{G}$, $\b\in k^{\times}$, $0\<r\<s$ and $t\>1$. Then
$$\begin{array}{rl}
V_{s+r}(\l)\ot V_{t}(\s,\b)&\cong(s-r)V_{t}(\s\l,\b)\oplus r V_{t-1}(\s\l,\b)\oplus r V_{t+1}(\s\l,\b),\\
V_{t}(\s,\b)\ot V_{s+r}(\l)&\cong(s-r)V_{t}(\s\l,{\l(a)}^{\ol s}\b)
\oplus r V_{t-1}(\s\l,{\l(a)}^{\ol s}\b)\oplus r V_{t+1}(\s\l,{\l(a)}^{\ol s}\b).\\
\end{array}$$
\end{lemma}
\begin{proof}
We only prove the first isomorphism since the second one can be shown similarly.
Let $M=V_{s+r}(\l)\ot V_{t}(\s,\b)$, and let $\varphi: M\ra M$ be the endomorphism of $M$
given by $\varphi(m)=x^{\ol s}m$, $m\in M$.
Let $\{m_i|0\<i\<s+r-1\}$ and $\{v_j|0\<j\<t\ol{s}-1\}$ be the standard bases of
 $V_{s+r}(\l)$ and $ V_{t}(\s,\b)$, respectively.
Then $\{m_i\ot v_j|0\<i\<s+r-1, 0\<j\<t{\ol s}-1\}$ is a basis of $M$.
Obviously, $\Pi(M)=\{\l\s, \chi\l\s, \cdots, \chi^{{\ol s}-1}\l\s\}$.
Moreover, $M_{(\chi^l\l\s)}={\rm span}\{m_i\ot v_j|0\<i\<s+r-1,\ 0\<j\<t{\ol s}-1,\ i+j\equiv l\ ({\rm mod}\ {\ol s})\}$
for all $0\<l\<{\ol s}-1$. In particular,
$$M_{(\chi^{{\ol s}-1}\l\s)}={\rm span}\{m_i\ot v_{j{\ol s}-i-1}| 1\<j\<t, 0\<i\<s+r-1\}.$$
Let $V_i={\rm span}\{m_{i}\ot v_{j{\ol s}-i-1}, m_{i+s}\ot v_{j{\ol s}-s-i-1}|1\<j\<t\}$ for $0\< i\< r-1$ (if $r>0$) and
$V_i={\rm span}\{m_{i}\ot v_{j{\ol s}-i-1}|1\<j\<t\}$ for $r\<i\<s-1$ (if $r<s$).
Then $M_{(\chi^{{\ol s}-1}\l\s)}=V_0\oplus V_1\oplus\cdots\oplus V_{s-1}$ as vector spaces and
$\varphi(V_i)\subseteq V_i$. For $0\< i\< r-1$ (if $r>0$),
a straightforward computation shows that under the basis
$m_i\ot v_{{\ol s}-i-1}, m_i\ot v_{2{\ol s}-i-1}, \cdots, m_i\ot v_{t{\ol s}-i-1}, m_{i+s}\ot v_{{\ol s}-s-i-1},
m_{i+s}\ot v_{2{\ol s}-s-i-1}, \cdots,
m_{i+s}\ot v_{t{\ol s}-s-i-1}
$ of $V_i$,  the matrix of the restriction $\varphi|_{V_i}$ is
$C=\left(
    \begin{array}{cc}
      A & 0 \\
      \a A & A \\
    \end{array}
  \right)
$,
where
$$A=\left(\begin{array}{ccccc}
0 &  0 & \cdots & 0 & \a_0 \\
1 & 0 & \cdots & 0 & \a_1\\
\vdots & \ddots & \ddots & \vdots & \vdots  \\
0 &  0 & \ddots & 0 & \a_{t-2}\\
0 & 0 & \cdots & 1 & \a_{t-1} \\
\end{array}\right),$$\\
$\a=s'{\s(a)}^{s}$ and $\a_j=(-1)^{t+1-j}\binom{t}{j}\b^{t-j}$ for $0\<j\<t-1$.
Similarly, under the basis $m_i\ot v_{{\ol s}-i-1}, m_i\ot v_{2{\ol s}-i-1},
\cdots, m_i\ot v_{t{\ol s}-i-1}$ of $V_i$, the matrix of the
restriction $\varphi|_{V_i}$ is  the matrix $A$ for $r\<i\<s-1$ (if $r<s$).
Hence the matrix of the restriction $\varphi|_{M_{(\chi^{{\ol s}-1}\s\l)}}$ (under some suitable basis) is
$$D=\left(\begin{array}{cccccc}
C&&&&&\\
&\ddots&&&&\\
&&C&&&\\
&&&A&&\\
&&&&\ddots&\\
&&&&&A\\
\end{array}\right).$$\\
with $r$ copies of $C$ and $s-r$ copies of $A$.
One can check that ${\rm det}(D)\neq0$.
A tedious but standard computation shows that the Jordan form of $C$
is $\left(\begin{array}{cc}
J_{t-1}(\b)&0\\
0&J_{t+1}(\b)\\
\end{array}\right)$, and the Jordan form of $A$ is $J_{t}(\b)$.
Thus it follows from \cite[Lemma 3.4]{SunChen} that
$V_{s+r}(\l)\ot V_{t}(\s,\b)$ contains a submodule isomorphic to
$(s-r)V_{t}(\chi^{{\ol s}-1}\s\l,\b)\oplus rV_{t-1}(\chi^{{\ol s}-1}\s\l,\b)\oplus rV_{t+1}(\chi^{{\ol s}-1}\s\l,\b)$
and so $V_{s+r}(\l)\ot V_{t}(\s,\b)\cong(s-r)V_{t}(\s\l,\b)\oplus rV_{t-1}(\s\l,\b)\oplus rV_{t+1}(\s\l,\b)$
since the modules on the both sides have the same dimension $(s+r)t {\ol s}$.
\end{proof}

\begin{corollary}\colabel{3.11}
Let $\s,\l \in\hat{G}$, $\b\in k^{\times}$,  $t\>1$. Then
$$\begin{array}{rl}
V_{2s}(\l)\ot V_{t}(\s,\b)&\cong sV_{t-1}(\s\l,\b)\oplus sV_{t+1}(\s\l,\b),\\
V_{t}(\s,\b)\ot V_{2s}(\l)&\cong sV_{t-1}(\s\l,{\l(a)}^{\ol s}\b)\oplus sV_{t+1}(\s\l,{\l(a)}^{\ol s}\b).\\
\end{array}$$
In particular, we have
$$V_{2s}(\e)\ot V_{t}(\s,\b)\cong V_{t}(\s,\b)\ot V_{2s}(\e)\cong sV_{t-1}(\s,\b)\oplus sV_{t+1}(\s,\b).$$
\end{corollary}

\begin{proposition}\prlabel{3.12}
Let $n, t\in\mathbb{Z}$ with $n, t\>1$, $\l, \s\in\hat{G}$ and $\b\in k^{\times}$.
Let $n=us+r$ with $u\>0$ and $0\<r<s$. Then
$$\begin{array}{rl}
V_n(\l)\ot V_t(\s,\b)
\cong&(\oplus_{1\<i\<{\rm min}\{t,u\}}(s-r)V_{2i-1+|t-u|}(\s\l,\b))\\
&\oplus(\oplus_{i=1}^{{\rm min}\{t,u+1\}}rV_{2i-1+|t-u-1|}(\s\l,\b)),\\
V_t(\s,\b)\ot V_n(\l)
\cong&(\oplus_{1\<i\<{\rm min}\{t, u\}}(s-r)V_{2i-1+|t-u|}(\s\l,\l(a)^{\ol s}\b))\\
&\oplus(\oplus_{i=1}^{{\rm min}\{t, u+1\}}rV_{2i-1+|t-u-1|}(\s\l,\l(a)^{\ol s}\b)).\\
\end{array}$$
\end{proposition}

\begin{proof}
Similarly to \cite[Theorem 3.6]{SunChen}, it can be shown  for $n\<s$ and $n>s$, respectively.
For $n\<s$, the proof is similar to Case 1 in the proof of \cite[Theorem 3.6]{SunChen}.
For $n>s$, the proposition can be shown by induction on $t$.
We only consider the case of $t=1$ since the arguments are similar to Case 2 in the proof of \cite[Theorem 3.6]{SunChen}
for $t=2$ and the induction step.

Let $n\>s$. We prove the decomposition of $V_n(\l)\ot V_1(\s,\b)$
by induction on $u$ for $r=0$ and $r>0$, respectively.
Note that $u\>1$ by $n\>s$. First assume that $r=0$. Then
for $u=1, 2$, it follows from \leref{3.10}. Now let $u\>2$.
Then by the induction hypothesis, \leref{3.5} and \leref{3.10}, we have
$$\begin{array}{rl}
&V_{s+1}(\e)\ot V_{us}(\l) \ot V_1(\s,\b)\\
\cong&V_{s+1}(\e)\ot sV_{u}(\s\l,\b)\\
\cong& s(s-1)V_{u}(\s\l,\b)\oplus sV_{u-1}(\s\l,\b)\oplus sV_{u+1}(\s\l,\b)\\
\end{array}$$
and
$$\begin{array}{rl}
&V_{s+1}(\e)\ot V_{us}(\l) \ot V_1(\s,\b)\\
\cong&(V_{(u+1)s}(\l)\oplus (\oplus_{i=1}^{s-1}V_{us}(\chi^{i}\l))\oplus V_{(u-1)s}(\chi^{s}\l))\ot V_1(\s,\b)\\
\cong& V_{(u+1)s}(\l)\ot V_1(\s,\b)\oplus (\oplus_{i=1}^{s-1}sV_{u}(\chi^{i}\s\l,\b))\oplus sV_{u-1}(\chi^{s}\s\l,\b).\\
\end{array}$$
Thus, it follows from Krull-Schmidt Theorem that $V_{(u+1)s}(\l)\ot V_1(\s,\b)\cong sV_{u+1}(\s\l,\b)$.
Next assume that $1\< r \<s-1$. Then for $u=1$, it follow from  \leref{3.10}.
Now let $u\>1$. Then by the induction hypothesis, \leref{3.5}, \leref{3.10},
the result above for $r=0$ and the result for $n\<s$, we have
$$\begin{array}{rl}
&V_{s+1}(\e)\ot V_{us+r}(\l) \ot V_1(\s,\b)\\
\cong&V_{s+1}(\e)\ot((s-r)V_{u}(\s\l,\b)\oplus rV_{u+1}(\s\l,\b))\\
\cong& (s-r)(s-1)V_{u}(\s\l,\b)\oplus (s-r)V_{u-1}(\s\l,\b)\oplus (s-r)V_{u+1}(\s\l,\b)\\
&\oplus r(s-1)V_{u+1}(\s\l,\b)\oplus rV_u(\s\l,\b) \oplus rV_{u+2}(\s\l,\b)
\end{array}$$
and
$$\begin{array}{rl}
&V_{s+1}(\e)\ot V_{us+r}(\l) \ot V_1(\s,\b)\\
\cong&(V_{(u+1)s+r}(\l)\oplus(\oplus_{1\<i\<r-1}V_{(u+1)s}(\chi^{i}\l))\oplus V_{(u+1)s-r}(\chi^r\l)
\oplus(\oplus_{r+1\<i\<s-1}V_{us}(\chi^{i}\l)) \\
&\oplus V_{(u-1)s+r}(\chi^{s}\l))\ot V_1(\s,\b)\\
\cong& V_{(u+1)s+r}(\l)\ot V_1(\s,\b) \oplus (\oplus_{1\<i\<r-1}sV_{u+1}(\chi^i\s\l,\b))
\oplus (s-r)V_{u+1}(\chi^r\s\l,\b)\\
&\oplus rV_u(\chi^r\s\l,\b)\oplus (\oplus_{r+1\<i\<s-1}sV_{u}(\chi^i\s\l,\b))\oplus (s-r)V_{u-1}(\chi^{s}\s\l,\b)
\oplus rV_u(\chi^{s}\s\l,\b)\\
\cong& V_{(u+1)s+r}(\l)\ot V_1(\s,\b) \oplus (r-1)sV_{u+1}(\s\l,\b)\oplus rV_u(\s\l,\b)\oplus (s-r)V_{u+1}(\s\l,\b)\\
&\oplus (s-1-r)sV_{u}(\s\l,\b)\oplus (s-r)V_{u-1}(\s\l,\b)\oplus rV_u(\s\l,\b).\\
\end{array}$$\\
Thus, it follows from Krull-Schmidt Theorem that
$$V_{(u+1)s+r}(\l)\ot V_1(\s,\b)\cong (s-r)V_{(u+1)}(\s\l,\b)\oplus rV_{u+2}(\s\l,\b).$$
The decomposition of $V_1(\s,\b)\ot V_n(\l)$ can be shown similarly.
\end{proof}

\subsection{Tensor products $V_n(\s, \a)\ot V_t(\l, \b)$}\selabel{3.2}
In this subsection, we consider the tensor products of two module of non-nilpotent type.
Throughout the following, let $\xi\in k$ be a  root of unity of order $s'$.

\begin{notation}\nolabel{3.0}
For  $\a, \b\in k^{\times}$, let $\theta, \eta\in k$ be two scales satisfying $\theta^{s'}=\a$
and $\eta^{s'}=\b$, respectively. Then for $\l\in\hat{G}$ and $1\<i, j\<s'$, let
$\a_{ij}=\theta\l^{s}(a)\xi^{i-1}+ \eta\xi^{j-1}$. Then $\a_{ij}^{s'}=\a_{mn}^{s'}$ if $j-i\equiv n-m$ (mod $s'$).
Denote $\a_{1j}^{s'}$ by $\a_j$ for all $1\<j\<s'$.
\end{notation}

\begin{lemma}\lelabel{3.13}
Let $\l , \s \in\hat{G}$ and $\a, \b\in k^{\times}$. With the notations given in \noref{3.0}, we have\\
{\rm (1)} if $\b+(-1)^{s'+1}\a{\l(a)}^{\ol s}\neq 0$, then $V(\s,\a)\ot V(\l,\b)\cong\oplus_{j=1}^{s'}sV(\s\l,\a_j)$;\\
{\rm (2)} if $\b+(-1)^{s'+1}\a{\l(a)}^{\ol s}=0$, then there is unique $j_0$ with $1\<j_0\<s'$ such that $\a_{j_0}=0$, and
$V(\s,\a)\ot V(\l,\b)\cong (\oplus_{1\<j\<s', j\neq j_0}sV(\s\l,\a_j))\oplus (\oplus_{j=0}^{{\ol s}-1}V_{s}(\chi^j\s\l))$.
\end{lemma}

\begin{proof}
Let $\{m_i|0\<i\<{\ol s}-1\}$ and $\{v_j|0\<j\<{\ol s}-1\}$
be the bases of $V(\s,\a)$ and $V(\l,\b)$ as stated in the last section, respectively.
Let $M=V(\s,\a)\ot V(\l,\b)$.
Then $\{m_i\ot v_j|0\<i, j\<{\ol s}-1\}$
is a basis of $M$. Moreover, $\Pi(M)=\{\chi^l\l\s|0\<l\<{\ol s}-1\}$
and for $0\<l\<{\ol s}-1$,
$$M_{(\chi^l\l\s)}={\rm span}\{m_i\ot v_j|0\<i, j\<{\ol s}-1, i+j\equiv l\ ({\rm mod}\ {\ol s})\}.$$
Let $\varphi: M\ra M$ be the linear endomorphisms of $M$ defined by $\varphi(m)=xm$, $m\in M$.
Then $\varphi(M_{(\chi^i\s\l)})\subseteq M_{(\chi^{i+1}\s\l)}$ for $0\<i\<{\ol s}-1$ since $M_{(\chi^{\ol s}\s\l)}=M_{\s\l}$.
Hence ${\rm Ker}(\varphi)=\oplus_{i=0}^{{\ol s}-1}({\rm Ker}(\varphi)\cap M_{(\chi^i\s\l)})$.
For any $0\<i\<{\ol s}-2$, let $z=\sum_{j=0}^i\g_jm_j\ot v_{i-j}+\sum_{j=i+1}^{{\ol s}-1}\g_jm_j\ot v_{{\ol s}+i-j} \in M_{(\chi^i\s\l)}$
for some $\g_0, \g_1, \cdots, \g_{{\ol s}-1}\in k$. By a straightforward computation, one gets that
$$\begin{array}{rl}
\varphi(z)=&(\g_0+\a q^{{\ol s}-i-1}\l(a)\g_{{\ol s}-1})m_0\ot v_{i+1}+\sum_{1\<j\<i}(\g_j+q^{j-i-1}\l(a)\g_{j-1})m_j\ot v_{i+1-j}\\
&+(\l(a)\g_i+\b\g_{i+1})m_{i+1}\ot v_0+\sum_{i+2\<j\<{\ol s}-1}(\g_j+\g_{j-1}q^{j-i-1}\l(a))m_j\ot v_{{\ol s}+i+1-j}.\\
\end{array}$$
Thus, $\varphi(z)=0$ if and only if ($\g_0, \g_1, \cdots, \g_{{\ol s}-1}$) is a solution of following system of linear equations
$$\left\{\begin{array}{ccc}
    x_0+\a q^{{\ol s}-i-1}\l(a)x_{{\ol s}-1}&=&0 \\
    q^{-i}\l(a)x_0+x_1&=&0 \\
    q^{1-i}\l(a)x_1+x_2&=&0 \\
    \cdots\ \cdots\  \cdots&& \\
     q^{i-1-i}\l(a)x_{i-1}+x_i&=&0 \\
    \l(a)x_i+\b x_{i+1}&=&0\\
    q\l(a)x_{i+1}+x_{i+2}&=&0\\
    q^2\l(a)x_{i+2}+x_{i+3}&=&0\\
     \cdots\ \cdots\  \cdots&&\\
    q^{{\ol s}-i-2}\l(a)x_{{\ol s}-2}+x_{{\ol s}-1}&=&0\\
  \end{array}\right.$$
Let $D$ be the coefficient matrix of the system of linear equations, and ${\rm r}(D)$ the rank of $D$.
Then ${\rm r}(D)\>{\ol s}-1$ and det($D$)=$\b+(-1)^{s'+1}\a\l(a)^{\ol s}$. Hence
${\rm Ker}(\varphi)\cap M_{(\chi^i\s\l)}=0$ if $\b+(-1)^{s'+1}\a\l(a)^{\ol s}\neq 0$,
and ${\rm dim}({\rm Ker}(\varphi)\cap M_{(\chi^i\s\l)})=1$ if $\b+(-1)^{s'+1}\a\l(a)^{\ol s}=0$,
where $0\<i\<{\ol s}-2$. Similarly, one can check that
${\rm Ker}(\varphi)\cap M_{(\chi^{{\ol s}-1}\s\l)}=0$ if $\b+(-1)^{s'+1}\a\l(a)^{\ol s}\neq 0$,
and ${\rm dim}({\rm Ker}(\varphi)\cap M_{(\chi^{{\ol s}-1}\s\l)})=1$ if $\b+(-1)^{s'+1}\a\l(a)^{\ol s}=0$.

Let $1\<l\<s'$. For any  $0\<i\<s-1$, define a subspace $V_i^l$ of  $M_{(\chi^{ls-1}\s\l)}$ by
$$V_i^l={\rm span}\{m_{js+i}\ot v_{(l-j)s-1-i}, m_{ts+i}\ot v_{(s'+l-t)s-1-i}|0\<j\<l-1, l\<t\<s'-1\}.$$
Then $\varphi^{\ol s}(V^l_i)\subseteq V^l_i$ and $M_{(\chi^{ls-1}\s\l)}=V^l_0\oplus V^l_1\oplus V^l_2\oplus \cdots V^l_{s-1}$.
For any $0\<i\<s-1$, let $C_{l,i}$ be the matrix of $\varphi^{\ol s}|_{V^l_i}$
under the basis $\{m_i\ot v_{ls-1-i}, m_{s+i}\ot v_{(l-1)s-1-i}, \cdots, m_{(l-1)s}\ot v_{s-1-i}, m_{ls+i}\ot v_{s's-1-i},
m_{(l+1)s+i}\ot v_{(s'-1)s-1-i}, \cdots, m_{(s'-1)s+i}\ot v_{(l+1)s-1-i}$\} of $V^l_i$.
A straightforward computation shows that $C_{l,0}=C_{l,1}=\cdots=C_{l,s-1}$, denoted by $C_l$ simply.
Hence the matrix of the restriction $\varphi^{\ol s}|_{M_{(\chi^{ls-1}\s\l)}}$ (under some suitable basis ) is
 $$A_l=\left(
   \begin{array}{cccc}
     C_l & 0 & \cdots & 0 \\
     0 & C_l & \cdots & 0 \\
     \cdots & \cdots & \cdots & \cdots \\
     0 & 0 & \cdots & C_l \\
   \end{array}
\right).$$
We claim that $C_l$ is diagonalizable, and so is $A_l$, where  $1\<l\<s'$. In fact,
let $V={\rm span}\{m_{is}\ot v_{js-1}|0\<i\<s'-1, 1\<j\<s'\}$. Then one can check that $\varphi^{s}(V)\subseteq V$,
and that under the basis $\{m_0\ot v_{s-1}, m_0\ot v_{2s-1}, \cdots, m_0\ot v_{s's-1},
m_{s}\ot v_{s-1}, m_{s}\ot v_{2s-1}, \cdots, m_{s}\ot v_{s's-1}, \cdots,
m_{(s'-1)s}\ot v_{s-1}, m_{(s'-1)s}\ot v_{2s-1}, \cdots, m_{(s'-1)s}\ot v_{s's-1}\}$ of $V$,
the matrix of the restriction $\varphi^{s}|_V$ is
 $$F=\left(
     \begin{array}{ccccc}
       B & 0 & \cdots & 0&\a\l(a)^{s}I_{s'}\\
       \l(a)^{s}I_{s'} & B & \cdots&0  & 0 \\
       \cdots & \cdots & \cdots  & \cdots&\cdots \\
       0&0&\cdots&B&0\\
       0 & 0&\cdots & \l(a)^{s}I_{s'} & B \\
     \end{array}
   \right)
 $$
where $B=\left(
           \begin{array}{cc}
             0 & \b \\
             I_{s'-1} &0 \\
           \end{array}
         \right)
$, $I_n$ denotes the $n\times n$ identity  matrix over $k$. Consider the matrix $yI_{{s'}^2}-F\in M_{{s'}^2}(k[y])$,
where $k[y]$ is the polynomial algebra over $k$ in one variable $y$.
Since $B$ is similar to the diagonal matrix ${\rm diag}\{\eta, \eta\xi, \cdots, \eta\xi^{s'-1}\}$,
$yI_{{s'}^2}-F$ is equivalent to the diagonal matrix
$${\rm diag}\{1, \cdots, 1, (y-\eta)^{s'}-\a\l(a)^{\ol s}, (y-\eta\xi)^{s'}-\a\l(a)^{\ol s}, \cdots,
(y-\eta\xi^{s'-1})^{s'}-\a\l(a)^{\ol s}\}.$$
It follows that $y-\a_{ij}$ gives all elementary divisors of $F$, where $1\<i, j\<s'$,
and so $F$ is diagonalizable. Hence $F^{s'}$ is also diagonalizable,
and $\a_{ij}^{s'}$ ($1\<i, j\<s'$) are its all eigenvalues.
Note that $V=V_0^1\oplus V_0^2\oplus \cdots V_0^{s'}$,
and $\varphi^{\ol s}(V_0^l)\subseteq V_0^l$ for all $1\<l\<s'$.
Therefore, $C_l$ is diagonalizable, and so is $A_l$.

For $1\<l\<s'$, let $\b_{l1}, \b_{l2}, \cdots, \b_{ls'}$ be the eigenvalues of $C_l$.
Then $\b_{lt}$, $1\<l, t\<s'$ are all eigenvalues of $F^{s'}$.
Hence there is a permutation $\pi$ of the set $\{(i,j)|1\<i,j\<s'\}$ such that
$\b_{lt}=a_{\pi(lt)}^{s'}$ for all $1\<l, t\<s'$.

(1) Assume that $\b+(-1)^{s'+1}\a{\l(a)}^{\ol s}\neq 0$. Then $\varphi$ is bijective and ${\rm det}(A_l)\neq 0$
for any $1\<l\<s'$.
It follows from \cite[Lemma 3.4]{SunChen} that $M$ contains a submodule isomorphic to
$\oplus_{j=1}^{s'}sV(\chi^{{\ol s}-1}\s\l, \b_{lj})\cong \oplus_{j=1}^{s'}sV(\s\l, \b_{lj})$,
and so $M\cong\oplus_{j=1}^{s'}sV(\s\l, \b_{lj})$ since they have the same dimension.
Now let $2\<l\<s'$. Then we have $M\cong \oplus_{j=1}^{s'}sV(\s\l, \b_{lj})\cong \oplus_{j=1}^{s'}sV(\s\l, \b_{1j})$.
Thus, by Krull-Schmidt Theorem, one knows that there is a permutation $\pi_l$ of the set $\{1, 2, \cdots, s'\}$
such that $\b_{lj}=\b_{1\pi_l(j)}$ for all $1\<j\<s'$. Since $\a_{ij}^{s'}=\a_{lt}^{s'}$ for any $1\<i, j, l, t\<s'$ with $j-i\equiv l-t$ (mod $s'$),
there is a permutation $\pi_0$ of $\{1, 2, \cdots, s'\}$ such that $\b_{1j}=\a_{1\pi_0(j)}^{s'}=\a_{\pi_0(j)}$
for all $1\<j\<s'$. It follows that $M\cong \oplus_{j=1}^{s'}sV(\s\l, \b_{1j})\cong \oplus_{j=1}^{s'}sV(\s\l, \a_j)$.

(2) Assume that $\b+(-1)^{s'+1}\a{\l(a)}^{\ol s}=0$. Then ${\rm Ker}(\varphi)\cap M_{(\chi^i\s\l)}\neq 0$
for all $0\<i\<{\ol s}-1$, and hence ${\rm Ker}(\varphi^{\ol s}|_{M_{(\chi^{ls-1}\s\l)}})\neq 0$ for all $1\<l\<s'$.
This implies that ${\rm det}(A_l)=0$, and so ${\rm det}(C_l)=0$.
Hence $0$ is an eigenvalue of $C_l$.
We claim that the multiplicity of eigenvalue $0$  of $C_l$ is 1  for any $1\<l\<s'$.
In fact, ${\rm det}(F)=0$ by ${\rm det}(F)^{s'}={\rm det}(C_1) {\rm det}(C_2)\cdots{\rm det}(C_{s'})=0$.
Hence $0$ is an eigenvalue of $F$. By the discussion before, $\a_{ij}$ and $\a_{ij}^{s'}$ are all eigenvalue of $F$
and $F^{s'}$, respectively, where $1\<i,j\<s'$. For any $1\<i,j\<s'$, we have $\a_{ij}=\a_{1t}\xi^{i-1}$,
where $t=j-i+1$ if $i\<j$, and $t=s'+j-i+1$ if $i>j$.  Moreover, $\a_{11}, \a_{12}, \cdots, \a_{1s'}$
are distinct. It follows that  the multiplicity of eigenvalue $0$  of $F$ is $s'$, and so is that of  $F^{s'}$.
Hence there is a unique $j_0$ with $1\<j_0\<s'$ such that $\a_{1j_0}=0$ and $\a_{j_0}=0$.
Since $F^{s'}$ is similar to
$$\left(
\begin{array}{cccc}
C_1 & 0 & \cdots & 0 \\
0 & C_2 & \cdots & 0 \\
\vdots & \vdots & \ddots & \vdots \\
0 & \cdots & \cdots & C_{s'} \\
\end{array}
\right),$$
the multiplicity of eigenvalue $0$  of each $C_l$ is 1. Thus,  for any $1\<l\<s'$,
there exists an integer $j_l$ with $1\<j_l\<s'$ such that $\b_{lj_l}=0$ and $\b_{lj}\neq 0$
if $j\neq j_l$, where $1\<j\<s'$. It follows from \cite[Lemma 3.4]{SunChen} that $M$ contains a submodule $N_l$
with $N_l\cong\oplus_{1\<j\<s', j\neq j_l}sV(\chi^{{\ol s}-1}\s\l, \b_{lj})\cong \oplus_{1\<j\<s', j\neq j_l}sV(\s\l, \b_{lj})$.

Now let $0\<i\<{\ol s}-1$. Then $i=ns+r$ with $0\<r\<s-1$ and $0\<n\<s'-1$. Define $z_i\in M_{(\chi^i\s\l)}$ by
$z_i=\sum_{j=0}^n(-1)^j\l(a)^{js}m_{js}\ot v_{(n-j)s+r}+\b^{-1}\sum_{j=n+1}^{s'-1}(-1)^j\l(a)^{js}m_{js}\ot v_{(s'+n-j)s+r}$
for $0\<n\<s'-2$ (or equivalently, $0\<i\<{\ol s}-s-1$), and
$z_i=\sum_{j=0}^{s'-1}(-1)^j\l(a)^{js}m_{js}\ot v_{(s'-1-j)s+r}$
for $n=s'-1$ (or equivalently, ${\ol s}-s\<i\<{\ol s}-1$).
Then a straightforward computation shows that $\varphi^{s-1}(z_i)\neq 0$ but $\varphi^{s}(z_i)=0$
by $\b+(-1)^{s'+1}\a\l(a)^{\ol s}=0$.  It follows that the submodule $\langle z_i\rangle$ of $M$
generated by $z_i$ is isomorphic to $V_s(\chi^i\s\l)$. Obviously, the sum $\sum_{i=0}^{{\ol s}-1}\langle z_i\rangle$
is direct. Hence $M$ contains a submodule $U$ with $U\cong\oplus_{i=0}^{{\ol s}-1}V_s(\chi^i\s\l)$.
Since every simple submodule of $U$ is 1-dimensional and every simple submodule of $N_l$
is $\ol s$-dimensional, the sum $N_l+U$ is direct, and consequently, $M=N_l\oplus U$
by comparing their dimensions, where $1\<l\<s'$. It follows from Krull-Schmidt Theorem that
$N_l\cong N_1$, and so
$$\oplus_{1\<j\<s', j\neq j_l}V(\s\l, \b_{lj})\cong \oplus_{1\<j\<s', j\neq j_1}V(\s\l, \b_{1j}),$$
for any $2\<l\<s'$. Thus, for any  $2\<l\<s'$, there is a permutation $\pi_l$ of the set $\{1, 2, \cdots, s'\}$
with $\pi_l(j_l)=j_1$ such that $\b_{lj}=\b_{1\pi_l(j)}$ for all $1\<j\<s'$.
Then by the discussion in (a), one can see that
$\oplus_{1\<j\<s', j\neq j_1}V(\s\l, \b_{1j})\cong \oplus_{1\<j\<s', j\neq j_0}V(\s\l, \a_j)$.
Therefore,
$$M=N_1\oplus U\cong (\oplus_{1\<j\<s', j\neq j_0}sV(\s\l, \a_j))\oplus(\oplus_{j=0}^{{\ol s}-1}V_{s}(\chi^j\s\l)).$$
\end{proof}

\begin{theorem}\thlabel{3.14}
Let $n, t\in\mathbb{Z}$ with $n, t\>1$, $\s, \l\in\hat{G}$ and $\a, \b\in k^{\times}$.
With the notations given in \noref{3.0}, we have\\
{\rm (1)} if $\b+(-1)^{s'+1}\a{\l(a)}^{\ol s}\neq 0$, then
$$V_n(\s,\alpha)\ot V_t(\l,\b)\cong\oplus_{i=1}^{{\rm min}\{n,t\}}\oplus_{j=1}^{s'}sV_{2i-1+|n-t|}(\s\l,\a_j);$$
{\rm (2)} if $\b+(-1)^{s'+1}\a{\l(a)}^{\ol s}= 0$, then
$$\begin{array}{rl}
&V_n(\s,\alpha)\ot V_t(\l,\b)\\
\cong&(\oplus_{i=1}^{{\rm min}\{n,t\}}\oplus_{1\<j\<s',j\neq j_0}sV_{2i-1+|n-t|}(\s\l,\a_j))
\oplus(\oplus_{i=1}^{{\rm min}\{n,t\}}\oplus_{j=0}^{{\ol s}-1}V_{(2i-1+|n-t|)s}(\chi^j\s\l)),\\
\end{array}$$
where $1\<j_0\<s'$ with $\a_{j_0}=0$ as given in \leref{3.13}.
\end{theorem}

\begin{proof}
We only prove (2) since the proof is similar for (1).
Assume $\b+(-1)^{s'+1}\a{\l(a)}^{\ol s}=0$. We work by induction on $n$.
For $n=1$, we work by induction on $t$. If $t=1$, then it follows from \leref{3.13}.
For $t=2$, by \leref{3.6}, \coref{3.11} and \leref{3.13}, we have
$$V(\s,\a)\ot V(\l,\b)\ot V_{2s}(\e)\cong sV(\s,\a)\ot V_2(\l,\b)$$\\
and
$$\begin{array}{rl}
&V(\s,\a)\ot V(\l,\b)\ot V_{2s}(\e)\\
\cong&((\oplus_{1\<j\<s', j\neq j_0}sV(\s\l, \a_j))
\oplus(\oplus_{j=0}^{{\ol s}-1}V_{s}(\chi^j\s\l)))\ot V_{2s}(\e)\\
\cong&(\oplus_{1\<j\<s', j\neq j_0}s^2V_2(\s\l, \a_j))
\oplus(\oplus_{j=0}^{{\ol s}-1}\oplus_{i=0}^{s-1}V_{2s}(\chi^{i+j}\s\l))\\
\cong&(\oplus_{1\<j\<s', j\neq j_0}s^2V_2(\s\l, \a_j)
\oplus(\oplus_{j=0}^{{\ol s}-1}sV_{2s}(\chi^{j}\s\l)).\\
\end{array}$$
Then it follows from Krull-Schmidt Theorem that
$$V(\s,\a)\ot V_2(\l,\b)\cong(\oplus_{1\<j\<s', j\neq j_0}sV_2(\s\l, \a_j))
\oplus(\oplus_{j=0}^{{\ol s}-1}V_{2s}(\chi^j\s\l)).$$
Now let $t\>2$. Then by \leref{3.6}, \coref{3.11} and the induction hypothesis, we have
$$\begin{array}{rl}
&V(\s,\a)\ot V_t(\l,\b)\ot V_{2s}(\e)\\
\cong& sV(\s,\a)\ot V_{t-1}(\l,\b)\oplus sV(\s,\a)\ot V_{t+1}(\l,\b)\\
\cong&(\oplus_{1\<j\<s',j\neq j_0}s^2V_{t-1}(\s\l,\a_j))
\oplus(\oplus_{j=0}^{{\ol s}-1}sV_{(t-1)s}(\chi^j\s\l))
\oplus sV(\s,\a)\ot V_{t+1}(\l,\b)\\
\end{array}$$
and
$$\begin{array}{rl}
&V(\s,\a)\ot V_t(\l,\b)\ot V_{2s}(\e)\\
\cong&((\oplus_{1\<j\<s', j\neq j_0}sV_t(\s\l, \a_j))
\oplus(\oplus_{j=0}^{{\ol s}-1}V_{ts}(\chi^j\s\l)))\ot V_{2s}(\e)\\
\cong&(\oplus_{1\<j\<s', j\neq j_0}s^2(V_{t-1}(\s\l, \a_j)\oplus V_{t+1}(\s\l, \a_j)))\\
&\oplus(\oplus_{j=0}^{{\ol s}-1}\oplus_{p=0}^1\oplus_{i=0}^{s-1}V_{(t+1-2p)s}(\chi^{j+ps+i}\s\l))\\
\cong&(\oplus_{1\<j\<s', j\neq j_0}s^2(V_{t-1}(\s\l, \a_j)\oplus V_{t+1}(\s\l, \a_j)))\\
&\oplus(\oplus_{j=0}^{{\ol s}-1}\oplus_{i=0}^{s-1}(V_{(t+1)s}(\chi^{j+i}\s\l)\oplus V_{(t-1)s}(\chi^{j+s+i}\s\l)))\\
\cong&(\oplus_{1\<j\<s', j\neq j_0}s^2(V_{t-1}(\s\l, \a_j)\oplus V_{t+1}(\s\l, \a_j)))\\
&\oplus(\oplus_{j=0}^{{\ol s}-1}sV_{(t+1)s}(\chi^j\s\l))\oplus(\oplus_{j=0}^{{\ol s}-1}sV_{(t-1)s}(\chi^j\s\l)).\\
\end{array}$$
Then it follows from Krull-Schmidt Theorem that
$$V(\s,\a)\ot V_{t+1}(\l,\b)\cong(\oplus_{1\<j\<s', j\neq j_0}sV_{t+1}(\s\l, \a_j))
\oplus(\oplus_{j=0}^{{\ol s}-1}V_{(t+1)s}(\chi^j\s\l)).$$

For $n=2$, by \leref{3.6}, \coref{3.11} and the result for $n=1$ above, we have
$$V_{2s}(\e)\ot V(\s, \a)\ot V_t(\l,\b)\cong sV_{2}(\s,\a)\ot V_t(\l,\b)$$ \\
and
$$\begin{array}{rl}
V_{2s}(\e)\ot V(\s, \a)\ot V_t(\l,\b)
\cong&V_{2s}(\e)\ot((\oplus_{1\<j\<s', j\neq j_0}sV_t(\s\l, \a_j))
\oplus(\oplus_{j=0}^{{\ol s}-1}V_{ts}(\chi^j\s\l)))\\
\cong&(\oplus_{1\<j\<s', j\neq j_0}s^2(V_{t+1}(\s\l, \a_j)\oplus V_{t-1}(\s\l, \a_j)))\\
&\oplus(\oplus_{j=0}^{{\ol s}-1}\oplus_{i=0}^{s-1}(V_{(t+1)s}(\chi^{j+i}\s\l)\oplus V_{(t-1)s}(\chi^{j+i+s}\s\l)))\\
\cong&(\oplus_{1\<j\<s', j\neq j_0}s^2(V_{t+1}(\s\l, \a_j)\oplus V_{t-1}(\s\l, \a_j)))\\
&\oplus(\oplus_{j=0}^{{\ol s}-1}(sV_{(t+1)s}(\chi^{j}\s\l)\oplus sV_{(t-1)s}(\chi^{j}\s\l))),\\
\end{array}$$
and so
$$\begin{array}{rl}
&V_2(\s,\a)\ot V_t(\l,\b)\\
\cong&\oplus_{1\<j\<s', j\neq j_0}(sV_{t+1}(\s\l, \a_j)\oplus sV_{t-1}(\s\l, \a_j))
\oplus(\oplus_{j=0}^{{\ol s}-1}(V_{(t+1)s}(\chi^j\s\l)\oplus V_{(t-1)s}(\chi^j\s\l)))\\
\cong&(\oplus_{i=1}^{{\rm min}\{2,t\}}\oplus_{1\<j\<s',j\neq j_0}sV_{2i-1+|2-t|}(\s\l,\a_j))
\oplus(\oplus_{i=1}^{{\rm min}\{2,t\}}\oplus_{j=0}^{{\ol s}-1}V_{(2i-1+|2-t|)s}(\chi^j\s\l)).\\
\end{array}$$
Now let $n\>3$. By \leref{3.6}, \coref{3.11} and the induction hypothesis, we have
$$\begin{array}{rl}
&V_{2s}(\e)\ot V_{n-1}(\s,\a)\ot V_t(\l,\b)\\
\cong&sV_n(\s,\a)\ot V_t(\l,\b)\oplus sV_{n-2}(\s,\a)\ot V_t(\l,\b)\\
\cong&sV_n(\s,\a)\ot V_t(\l,\b)
\oplus(\oplus_{i=1}^{{\rm min}\{n-2,t\}}\oplus_{1\<j\<s', j\neq j_0}s^2V_{2i-1+|n-t-2|}(\s\l, \a_j))\\
&\oplus(\oplus_{i=1}^{{\rm min}\{n-2,t\}}\oplus_{j=0}^{{\ol s}-1}sV_{(2i-1+|n-t-2|)s}(\chi^j\s\l))\\
\end{array}$$
and
$$\begin{array}{rl}
&V_{2s}(\e)\ot V_{n-1}(\s, \a)\ot  V_t(\l, \b)\\
\cong&(\oplus_{i=1}^{{\rm min}\{n-1,t\}}\oplus_{1\<j\<s',j\neq j_0}sV_{2s}(\e)\ot V_{2i-1+|n-t-1|}(\s\l,\a_j))\\
&\oplus(\oplus_{i=1}^{{\rm min}\{n-1,t\}}\oplus_{j=0}^{{\ol s}-1}V_{2s}(\e)\ot V_{(2i-1+|n-t-1|)s}(\chi^j\s\l))\\
\cong&(\oplus_{i=1}^{{\rm min}\{n-1,t\}}\oplus_{1\<j\<s',j\neq j_0}s^2(V_{2i+|n-t-1|}(\s\l,\a_j)\oplus V_{2i-2+|n-t-1|}(\s\l,\a_j)))\\
&\oplus(\oplus_{i=1}^{{\rm min}\{n-1,t\}}\oplus_{j=0}^{{\ol s}-1}\oplus_{l=0}^{s-1}(V_{(2i+|n-t-1|)s}(\chi^{j+l}\s\l)
\oplus V_{(2i-2+|n-t-1|)s}(\chi^{j+l+s}\s\l)))\\
\cong&(\oplus_{i=1}^{{\rm min}\{n-1,t\}}\oplus_{1\<j\<s',j\neq j_0}s^2(V_{2i+|n-t-1|}(\s\l,\a_j)\oplus V_{2i-2+|n-t-1|}(\s\l,\a_j)))\\
&\oplus(\oplus_{i=1}^{{\rm min}\{n-1,t\}}\oplus_{j=0}^{{\ol s}-1}(sV_{(2i+|n-t-1|)s}(\chi^{j}\s\l)
\oplus sV_{(2i-2+|n-t-1|)s}(\chi^{j}\s\l))).\\
\end{array}$$
Then by a straightforward computation for $n-1<t$, $n-1=t$ and $n-1>t$, respectively,
it  follows from Krull-Schmidt Theorem that
$$\begin{array}{rl}
&V_n(\s,\alpha)\ot V_t(\l,\b)\\
\cong&(\oplus_{i=1}^{{\rm min}\{n,t\}}\oplus_{1\<j\<s',j\neq j_0}sV_{2i-1+|n-t|}(\s\l,\a_j))
\oplus(\oplus_{i=1}^{{\rm min}\{n,t\}}\oplus_{j=0}^{{\ol s}-1}V_{(2i-1+|n-t|)s}(\chi^j\s\l)).\\
\end{array}$$
This completes proof.
\end{proof}

\section{The Green ring of the category of weight modules}\selabel{4}

In this section, we investigate the Green ring $r(\mathcal W)$ of the tensor category $\mathcal W$, which is the subring of $r(H)$
generated by all finite dimensional weight modules over $H$.
By \cite[Lemma 3.3, Propositions 3.17 and 4.2]{WangYouChen}, the map
$\hat{G}\ra r(\mathcal W)$, $\l\mapsto [V_1(\l)]$, can be uniquely extended to
a ring monomorphism ${\mathbb Z}\hat{G}\ra r(\mathcal W)$.
Hence we may regard the group ring ${\mathbb Z}\hat{G}$ as a subring of $r(\mathcal W)$
under the identification $[V_1(\l)]=\l$, $\l\in\hat{G}$. We state the observation as follows.

\begin{lemma}\lelabel{4.1}
The commutative group ring ${\mathbb Z}\hat{G}$ is a subring of $r(\mathcal W)$
under the identification $[V_1(\l)]=\l$, $\l\in\hat{G}$.
\end{lemma}

Throughout the following, let $q=\chi^{-1}(a)$, and let $|q|$ denote the order of $q$. Then $|q|=\infty$ or $|q|=s<\infty$.
For any $V\in\mathcal{W}$ and $m\in\mathbb{Z}$ with $m\>0$, define $V^{\ot m}$ by
$V^{\ot 0}=V_1(\e)$ for $m=0$, $V^{\ot 1}=V$ for $m=1$, and $V^{\ot m}=V\ot V\ot\cdots\ot V$, the tensor product of
$m$-folds of $V$, for $m>1$. Note that $[V_1(\e)]=\e$ is the identity of the ring $r(\mathcal W)$.
For any real number $r$, let $[r]$ denote the integer part of $r$, i.e., $[r]$ is the maximal integer
with respect to $[r]\<r$. Let $\binom{0}{0}=1$.

{\bf Convention}: For a sum $r+\sum_{l\<i\<m}r_i$ of some elements in a ring $R$, we regard
$\sum_{l\<i\<m}r_i=0$ whenever $m<l$.

\begin{lemma}\lelabel{4.2}
Let $0\<m<|q|$. Then
$V_2(\epsilon)^{\otimes m}\cong \oplus_{i=0}^{[\frac{m}{2}]}\frac{m-2i+1}{m-i+1}\binom{m}{i}V_{m+1-2i}(\chi^i).$
\end{lemma}

\begin{proof}
We show the lemma by induction on $m$. The lemma holds obviously for $m=0$ and $m=1$.
Now let $1<m<|q|$. If $m=2l$ is even, then by the induction hypothesis,
\cite[Theorem 3.3, Lemma 3.10(1)]{SunChen} and the explanation before \leref{3.5},
we have

$$\begin{array}{rl}
V_2(\epsilon)^{\otimes m}&\cong V_2(\epsilon)\ot V_2(\epsilon)^{\otimes m-1}\\
&\cong\oplus_{i=0}^{l-1}\frac{2l-2i}{2l-i}\binom{2l-1}{i}V_2(\epsilon)\ot V_{2l-2i}(\chi^i)\\
&\cong\oplus_{i=0}^{l-1}\frac{2l-2i}{2l-i}\binom{2l-1}{i}(V_{2l+1-2i}(\chi^i)\oplus V_{2l-1-2i}(\chi^{i+1}))\\
&\cong(\oplus_{i=0}^{l-1}\frac{2l-2i}{2l-i}\binom{2l-1}{i}V_{2l+1-2i}(\chi^i))
\oplus(\oplus_{i=1}^{l}\frac{2l-2i+2}{2l-i+1}\binom{2l-1}{i-1}V_{2l+1-2i}(\chi^{i}))\\
&\cong V_{2l+1}(\epsilon)\oplus\frac{2}{l+1}\binom{2l-1}{l-1}V_1(\chi^{l})\oplus
(\oplus_{1\<i<l-1}(\frac{2l-2i}{2l-i}\binom{2l-1}{i}+\frac{2l-2i+2}{2l-i+1}\binom{2l-1}{i-1})V_{2l+1-2i}(\chi^{i}))\\
&\cong V_{2l+1}(\epsilon)\oplus\frac{1}{l+1}\binom{2l}{l}V_1(\chi^{l})\oplus
(\oplus_{1\<i\<l-1}\frac{2l-2i+1}{2l-i+1}\binom{2l}{i}V_{2l+1-2i}(\chi^{i}))\\
&\cong \oplus_{i=0}^{[\frac{m}{2}]}\frac{m-2i+1}{m-i+1}\binom{m}{i}V_{m+1-2i}(\chi^i).\\
\end{array}$$
If $m=2l+1$ is odd, then by the same reason as above, we have
$$\begin{array}{rl}
V_2(\epsilon)^{\otimes m}&\cong V_2(\epsilon)\ot V_2(\epsilon)^{\otimes m-1}\\
&\cong\oplus_{i=0}^l\frac{2l-2i+1}{2l-i+1}\binom{2l}{i}V_2(\epsilon)\ot V_{2l+1-2i}(\chi^i)\\
&\cong\oplus_{i=0}^l\frac{2l-2i+1}{2l-i+1}\binom{2l}{i}(V_{2l+2-2i}(\chi^i)\oplus V_{2l-2i}(\chi^{i+1}))\\
&\cong(\oplus_{i=0}^l\frac{2l-2i+1}{2l-i+1}\binom{2l}{i}V_{2l+2-2i}(\chi^i))
\oplus(\oplus_{i=1}^{l+1}\frac{2l-2i+3}{2l-i+2}\binom{2l}{i-1}V_{2l+2-2i}(\chi^{i}))\\
&\cong V_{2l+2}(\epsilon)\oplus
(\oplus_{i=1}^l(\frac{2l-2i+1}{2l-i+1}\binom{2l}{i}+\frac{2l-2i+3}{2l-i+2}\binom{2l}{i-1})V_{2l+2-2i}(\chi^{i}))\\
&\cong V_{2l+2}(\epsilon)\oplus(\oplus_{i=1}^l\frac{2l-2i+2}{2l-i+2}\binom{2l+1}{i}V_{2l+2-2i}(\chi^{i}))\\
&\cong\oplus_{i=0}^{[\frac{m}{2}]}\frac{m-2i+1}{m-i+1}\binom{m}{i}V_{m+1-2i}(\chi^i).\\
\end{array}$$
\end{proof}

\subsection{The case of $|\chi|=|q|=\infty$}\selabel{4.1}
In this subsection, we consider the case that $|\chi|=\infty$ and $q$ is not a root of unity.
Throughout this subsection, assume that $|\chi|=|q|=\infty$.
By \prref{2.1} and \cite[Theorem 3.3]{SunChen}, one knows that $r(\mathcal W)$ is a commutative ring.
Let $y=[V_2(\epsilon)]$ in $r(\mathcal W)$.

\begin{lemma}\lelabel{4.3}
Let $m\>1$. Then $[V_m(\epsilon)]=\sum_{i=0}^{[\frac{m-1}{2}]}(-1)^{i}\binom{m-1-i}{i}\chi^iy^{m-1-2i}$ in $r(\mathcal W)$.
\end{lemma}
\begin{proof}
We work by induction  $m$. For $m=1$ and $m=2$, it is obvious. Now let $m>2$. Then  by \cite[Theorem 3.3]{SunChen},
we have $V_2(\e)\ot V_{m-1}(\e)\cong V_m(\e)\oplus V_{m-2}(\chi)\cong V_m(\e)\oplus V_1(\chi)\ot V_{m-2}(\e)$.
Thus, by the induction hypothesis, one gets
$$\begin{array}{rl}
[V_m(\e)]&=[V_2(\e)][V_{m-1}(\e)]-[V_1(\chi)][V_{m-2}(\e)]\\
&=y(\sum_{i=0}^{[\frac{m-2}{2}]}(-1)^{i}\binom{m-2-i}{i}\chi^iy^{m-2-2i})
-\chi(\sum_{i=0}^{[\frac{m-3}{2}]}(-1)^{i}\binom{m-3-i}{i}\chi^iy^{m-3-2i})\\
&=\sum_{i=0}^{[\frac{m-2}{2}]}(-1)^{i}\binom{m-2-i}{i}\chi^iy^{m-1-2i}
+\sum_{i=1}^{[\frac{m-1}{2}]}(-1)^i\binom{m-2-i}{i-1}\chi^iy^{m-1-2i}\\
&=\sum_{i=0}^{[\frac{m-1}{2}]}(-1)^{i}\binom{m-1-i}{i}\chi^iy^{m-1-2i}.\\
\end{array}$$
\end{proof}

\begin{corollary}\colabel{4.4}
$r(\mathcal W)$ is generated, as a ring, by $\hat G$ and $y$.
Moreover, $\{\l y^t|t\>0,\l \in \hat{G}\}$	is a $\mathbb{Z}$-basis of $r(\mathcal W)$.
\end{corollary}

\begin{proof}
By \prref{2.1}, one knows that $\{[V_t(\l)]|t\>1,\l \in \hat{G}\}$ is a $\mathbb{Z}$-basis of $r(\mathcal W)$.
For any $t\>1$ and $\l \in \hat{G}$, by \cite[Theorem 3.3]{SunChen} and \leref{4.3}, we have
$$\begin{array}{rcl}
[V_t(\l)]&=&[V_1(\l)\ot V_t(\e)]=[V_1(\l)][V_t(\e)]\\
&=&\l\sum_{i=0}^{[\frac{t-1}{2}]}(-1)^{i}\binom{t-1-i}{i}\chi^iy^{t-1-2i}.\\
\end{array}$$
This shows that $r(\mathcal W)$ is generated, as a ring, by $\hat G$ and $y$.
Hence as a $\mathbb Z$-module, $r(\mathcal W)$ is generated by $\{\l y^t|t\>0, \l\in\hat{G}\}$.
Then by \cite[Theorem 3.3]{SunChen} and \leref{4.2}, we have	
$$\begin{array}{rcl}
V_1(\l)\ot V_2(\epsilon)^{\otimes t}&\cong&\oplus_{i=0}^{[\frac{t}{2}]}\frac{t-2i+1}{t-i+1}\binom{t}{i}V_1(\l)\ot V_{t+1-2i}(\chi^i)\\
&\cong&V_{t+1}(\l)\oplus(\oplus_{1\<i\<[\frac{t}{2}]}\frac{t-2i+1}{t-i+1}\binom{t}{i}V_{t+1-2i}(\chi^i\l)),\\
\end{array}$$
and so $\l y^t=[V_{t+1}(\l)]+\sum_{1\<i\<[\frac{t}{2}]}\frac{t-2i+1}{t-i+1}\binom{t}{i}[V_{t+1-2i}(\chi^i\l)]$
in $r(\mathcal W)$, where $\l\in\hat{G}$ and $t\>0$.
It follows that the set $\{{\l}y^t|t\>0,\l\in \hat G\}$ is linearly independent over $\mathbb{Z}$.
Therefore, $\{{\l}y^t|t\>0,\l\in \hat G\}$ is $\mathbb{Z}$-basis of $r(\mathcal W)$.
\end{proof}

\begin{corollary}\colabel{4.5}
$r(\mathcal W)$ is isomorphic to the polynomial algebra $\mathbb{Z}\hat{G}[y]$ over the group ring $\mathbb{Z}\hat{G}$
in one variable $y$.
\end{corollary}
\begin{proof}
It follows from \leref{4.1} and \coref{4.4}.
\end{proof}

\subsection{The case of $|q|<|\chi|=\infty$}\selabel{4.2}
In this subsection, we consider the case that $|\chi|=\infty$ and $q$ is a root of unity.
Throughout this subsection, assume $|\chi|=\infty$ and $|q|=s<\infty$.
By \prref{2.1} and \prref{3.8}, $r(\mathcal W)$ is a commutative ring.
Let $y=[V_2(\epsilon)]$ and $z=[V_{s+1}(\epsilon)]$ in $r(\mathcal W)$.
Note that \cite[Lemmas 3.8-3.11]{SunChen} hold in this case.

\begin{proposition}\prlabel{4.6}
$r(\mathcal W)$ is generated, as a ring, by $\hat{G}\cup\{y,z\} $.
\end{proposition}
\begin{proof}
Let $R$ be the subring of $r(\mathcal W)$ generated by $\hat{G}\cup\{y,z\} $. By \prref{2.1}, It is enough to show
$[V_t(\l)]\in R$ for all $t\>1$ and $\l \in \hat G$. We work by induction on $t$. Obviously, $[V_1(\l)]=\l\in R$.
By \cite[Lemma 3.8]{SunChen}, $[V_2(\l)]=[V_1(\l)\ot V_2(\e)]={\l}y\in R$. Now Let $2\<t\<s-1$
and assume $[V_m(\l)]\in R$ for any $1\<m\<t$ and $\l\in\hat{G}$.
Then by \cite[Lemma 3.8 and Lemma 3.10]{SunChen},
$[V_{t+1}(\l)]=y[V_t(\l)]-\chi[V_{t-1}(\l)]\in R$. This shows that $[V_t(\l)]\in  R$ for all $1\<t\<s$ and $\l\in \hat G$.
Since $[V_{s+1}(\epsilon)]=z\in R$,  it follows from \cite[Lemma 3.8]{SunChen} that
$[V_{s+1}(\l)]=\l z\in R$ for $\l\in \hat G$. Then a similar argument as above shows that $[V_t(\l)]\in R$ for $s+1\<t\<2s$
and $\l\in \hat G$. Thus, we have proven that $[V_t(\l)]\in  R$ for all $1\<t\<2s$ and $\l\in \hat G$.
Now let $t>2s$ and assume $[V_{m}(\l)]\in R$ for all $1\<m<t$ and  $\l\in \hat G$.
If $s|t$, then by \leref{3.5}(1), we have
$[V_{t}(\l)]=z[V_{t-s}(\l)]-[V_{t-2s}(\chi^s\l)]-\sum\limits_{i=1}^{s-1}[V_{t-s}(\chi^i\l)]\in R$.
If $s\nmid t$, then $t-s=rs+l$ with $r\>1$ and $1\<l\<s-1$. In this case, by \leref{3.5}(2), we have
$$\begin{array}{rcl}
V_{s+1}(\e)\ot V_{t-s}(\l)
&\cong&V_{t}(\l)\oplus(\oplus_{1\<i\<l-1}V_{(r+1)s}(\chi^i\l))\oplus V_{t-2l}(\chi^l\l)\\
&&\oplus(\oplus_{l+1\<i\<s-1}V_{rs}(\chi^i\l))\oplus V_{t-2s}(\chi^{s}\l),\\
\end{array}$$
and so
$$\begin{array}{rcl}
[V_{t}(\l)]
&=&z[V_{t-s}(\l)]-\sum_{1\<i\<l-1}[V_{(r+1)s}(\chi^i\l)]-[V_{t-2l}(\chi^l\l)]\\
&&-\sum_{l+1\<i\<s-1}[V_{rs}(\chi^i\l)]-[V_{t-2s}(\chi^{s}\l)]\in R.\\
\end{array}$$
This completes the proof.
\end{proof}

\begin{lemma}\lelabel{4.7}
In $r(\mathcal W)$, we have\\
{\rm (1)} $[V_m(\e)]=\sum_{i=0}^{[\frac{m-1}{2}]}(-1)^{i}\binom{m-1-i}{i}\chi^iy^{m-1-2i}$,  $1\<m\<s$;\\
{\rm (2)} $[V_{s+m}(\epsilon)]=\sum_{i=0}^{[\frac{m-1}{2}]}(-1)^{i}\binom{m-1-i}{i}\chi^iy^{m-1-2i}z-
	\sum_{i=0}^{[\frac{m-2}{2}]}(-1)^{i}\binom{m-2-i}{i}\chi^{i+1}y^{m-2-2i}[V_{s}(\epsilon)]$, $2\<m\<s$;\\
{\rm (3)} $[V_{ms}(\epsilon)]=\sum_{i=0}^{[\frac{m-2}{2}]}(-1)^{i}\binom{m-2-i}{i}\chi^{si}(z-\sum_{j=1}^{s-1}\chi^j)^{m-2-2i}[V_{2s}(\epsilon)]
	-\sum_{i=0}^{[\frac{m-3}{2}]}(-1)^{i}\binom{m-3-i}{i}\chi^{s(i+1)}(z-\sum_{j=1}^{s-1}\chi^j)^{m-3-2i}[V_{s}(\epsilon)]$,  $m\>3$.
\end{lemma}
\begin{proof}
(1) and (2) can be shown similarly to \leref{4.3} by induction on $m$ and using \cite[Lemma 3.10(1)]{SunChen}.
We prove (3) by induction on $m$. At first, by \leref{3.5}(1) and \cite[Lemma 3.8]{SunChen},
for any integer $k\>1$, we have
$$\begin{array}{rl}
[V_{(k+1)s}(\epsilon)]=&z[V_{ks}(\epsilon)]-\chi^{s}[V_{(k-1)s}(\epsilon)]-\sum_{i=1}^{s-1}\chi^i[V_{ks}(\epsilon)]\\
	=&(z-\sum_{j=1}^{s-1}\chi^j)[V_{ks}(\epsilon)]-\chi^s[V_{(k-1)s}(\epsilon)].
\end{array}$$
Putting $k=2$, one gets (3) for $m=3$. Then putting $k=3$, one gets (3) for $m=4$.
Now let $m\>4$. Then by the equation above (putting $k=m$) and the induction hypothesis, we have
$$\begin{array}{rl}
[V_{(m+1)s}(\epsilon)]=&(z-\sum_{j=1}^{s-1}\chi^j)[V_{ms}(\epsilon)]-\chi^s[V_{(m-1)s}(\epsilon)]\\
=&\sum_{i=0}^{[\frac{m-2}{2}]}(-1)^{i}\binom{m-2-i}{i}\chi^{si}(z-\sum_{j=1}^{s-1}\chi^j)^{m-1-2i}[V_{2s}(\e)]\\
&-\sum_{i=0}^{[\frac{m-3}{2}]}(-1)^{i}\binom{m-3-i}{i}\chi^{s(i+1)}(z-\sum_{j=1}^{s-1}\chi^j)^{m-2-2i}[V_{s}(\e)]\\
&-\sum_{i=0}^{[\frac{m-3}{2}]}(-1)^{i}\binom{m-3-i}{i}\chi^{s(i+1)}(z-\sum_{j=1}^{s-1}\chi^j)^{m-3-2i}[V_{2s}(\e)]\\
&+\sum_{i=0}^{[\frac{m-4}{2}]}(-1)^{i}\binom{m-4-i}{i}\chi^{s(i+2)}(z-\sum_{j=1}^{s-1}\chi^j)^{m-4-2i}[V_{s}(\e)]\\
=&\sum_{i=0}^{[\frac{m-2}{2}]}(-1)^{i}\binom{m-2-i}{i}\chi^{si}(z-\sum_{j=1}^{s-1}\chi^j)^{m-1-2i}[V_{2s}(\e)]\\
&+\sum_{i=1}^{[\frac{m-1}{2}]}(-1)^{i}\binom{m-2-i}{i-1}\chi^{si}(z-\sum_{j=1}^{s-1}\chi^j)^{m-1-2i}[V_{2s}(\e)]\\
&-\sum_{i=0}^{[\frac{m-3}{2}]}(-1)^{i}\binom{m-3-i}{i}\chi^{s(i+1)}(z-\sum_{j=1}^{s-1}\chi^j)^{m-2-2i}[V_{s}(\e)]\\
&-\sum_{i=1}^{[\frac{m-2}{2}]}(-1)^{i}\binom{m-3-i}{i-1}\chi^{s(i+1)}(z-\sum_{j=1}^{s-1}\chi^j)^{m-2-2i}[V_{s}(\e)]\\
=&\sum_{i=0}^{[\frac{m-1}{2}]}(-1)^{i}\binom{m-1-i}{i}\chi^{si}(z-\sum_{j=1}^{s-1}\chi^j)^{m-1-2i}[V_{2s}(\e)]\\
&-\sum_{i=0}^{[\frac{m-2}{2}]}(-1)^{i}\binom{m-2-i}{i}\chi^{s(i+1)}(z-\sum_{j=1}^{s-1}\chi^j)^{m-2-2i}[V_{s}(\e)].\\
\end{array}$$
\end{proof}

\begin{lemma}\lelabel{4.8}
In $r(\mathcal W)$, we have
$$\begin{array}{rl}
y^{s}=&(1+\chi)\sum_{i=0}^{[\frac{s-1}{2}]}(-1)^i\binom{s-1-i}{i}\chi^iy^{s-1-2i}\\
&+\sum_{1\<i\<[\frac{s-1}{2}]}\sum_{j=0}^{[\frac{s-2i-1}{2}]}(-1)^j\frac{s-2i}{s-i}
\binom{s-1}{i}\binom{s-2i-1-j}{j}\chi^{i+j}y^{s-2i-2j}.\\
\end{array}$$
\end{lemma}
\begin{proof}
By \leref{4.2}, we have
$V_2(\epsilon)^{\otimes (s-1)}\cong \oplus_{i=0}^{[\frac{s-1}{2}]}\frac{s-2i}{s-i}\binom{s-1}{i}V_{s-2i}(\chi^i)$.
Then by \cite[Lemma 3.10]{SunChen}, one gets that
$$\begin{array}{c}
V_2(\epsilon)^{\ot s}=V_s(\epsilon)\oplus V_s(\chi)
\oplus(\oplus_{1\<i\<[\frac{s-1}{2}]}\frac{s-2i}{s-i}\binom{s-1}{i}(V_{s-2i+1}(\chi^i)\oplus V_{s-2i-1}(\chi^{i+1}))).
\end{array}$$
Thus, it follow from \leref{4.7} that
$$\begin{array}{rl}
y^{s}=&(1+\chi)\sum_{i=0}^{[\frac{s-1}{2}]}(-1)^i\binom{s-1-i}{i}\chi^iy^{s-1-2i}\\
&+\sum\limits_{1\<i\<[\frac{s-1}{2}]}\sum_{j=0}^{[\frac{s-2i-1}{2}]}(-1)^j\frac{s-2i}{s-i}\binom{s-1}{i}\binom{s-2i-1-j}{j}\chi^{i+j}y^{s-2i-2j}.
\end{array}$$
\end{proof}

For any $m\>0$, define a set $P_m$ of $H$-modules by $P_0=\{0\}$ and $P_m=\{\oplus_{i=1}^{ms}n_iV_i(\l_i)|\l_i\in\hat{G}, n_i\>0, 1\<i\<ms\}$
if $m\>1$.

\begin{lemma}\lelabel{4.9}
$V_{s+1}(\e)^{\ot m}\cong V_{ms+1}(\epsilon)\oplus E_m$ for some $E_m\in P_m$, $m\>0$.
\end{lemma}
\begin{proof}
We prove the lemma by induction on $m$.
For $m=0, 1$, it is obvious.
Now let $m\>1$ and assume $V_{s+1}(\e)^{\ot m}\cong V_{ms+1}(\e)\oplus E_m$ for some $E_m\in P_m$.
Then by \leref{3.5}(2), one gets that
$$\begin{array}{rl}
V_{s+1}(\e)^{\ot(m+1)}\cong &V_{s+1}(\e)\ot V_{ms+1}(\e)\oplus V_{s+1}(\e)\ot E_m\\
\cong&V_{(m+1)s+1}(\e)\oplus V_{(m+1)s-1}(\chi)\oplus(\oplus_{2\<i\<s-1}V_{ms}(\chi^i))\\
&\oplus V_{(m-1)s+1}(\chi^{s})\oplus V_{s+1}(\e)\ot E_m.\\
\end{array}$$
Obviously, $V_{(m+1)s-1}(\chi)\oplus(\oplus_{2\<i\<s-1}V_{ms}(\chi^i))\oplus V_{(m-1)s+1}(\chi^{s})\in P_{m+1}$.
By \leref{3.5}, one can see that $V_{s+1}(\e)\ot E_m\in P_{m+1}$. This completes the proof.
\end{proof}

\begin{proposition}\prlabel{4.10}
$\{\l y^tz^m|0\<t\<s-1, m\>0, \l\in\hat{G}\}$ is a $\mathbb{Z}$-basis of $r(\mathcal W)$.
\end{proposition}
\begin{proof}
Since $r(\mathcal W)$ is a commutative ring, it follows from \prref{4.6} and \leref{4.8} that
$r(\mathcal W)$ is generated, as a $\mathbb Z$-module, by $\{\l y^tz^m|0\<t\<s-1, m\>0, \l\in\hat{G}\}$.
It is left to show that $\{\l y^tz^m|0\<t\<s-1, m\>0, \l\in\hat{G}\}$ is linearly independent over $\mathbb{Z}$.

Let $0\<t\<s-1$ and $m\>0$. Then by \leref{4.2}, \leref{4.9} and \cite[Lemma 3.11(2)]{SunChen}, we have
$$\begin{array}{rl}
&V_{2}(\e)^{\ot t}\ot V_{s+1}(\e)^{\ot m}\\
\cong &V_{2}(\e)^{\ot t}\ot (V_{ms+1}(\e)\oplus E_m)\\
\cong &\sum_{i=0}^{[\frac{t}{2}]}\frac{t-2i+1}{t-i+1}\binom{t}{i}V_{t+1-2i}(\chi^i)\ot V_{ms+1}(\e)\oplus V_{2}(\e)^{\ot t}\ot E_m\\
\cong &\sum_{i=0}^{[\frac{t}{2}]}\frac{t-2i+1}{t-i+1}\binom{t}{i}(V_{t+1-2i+ms}(\chi^i)\oplus(\oplus_{1\<j\<t-2i}V_{ms}(\chi^{i+j})))
\oplus V_{2}(\e)^{\ot t}\ot E_m\\
\cong &V_{t+1+ms}(\e)\oplus(\oplus_{1\<i\<[\frac{t}{2}]}\frac{t-2i+1}{t-i+1}\binom{t}{i}V_{t+1-2i+ms}(\chi^i))\\
&\oplus(\oplus_{i=0}^{[\frac{t}{2}]}\oplus_{1\<j\<t-2i}\frac{t-2i+1}{t-i+1}\binom{t}{i}V_{ms}(\chi^{i+j}))\oplus V_{2}(\e)^{\ot t}\ot E_m\\
\cong &V_{ms+t+1}(\e)\oplus(\oplus_{1\<i\<[\frac{t}{2}]}\frac{t-2i+1}{t-i+1}\binom{t}{i}V_{t+1-2i+ms}(\chi^i))\oplus E,\\
\end{array}$$
where $E=(\oplus_{i=0}^{[\frac{t}{2}]}\oplus_{1\<j\<t-2i}\frac{t-2i+1}{t-i+1}\binom{t}{i}V_{ms}(\chi^{i+j}))\oplus V_{2}(\e)^{\ot t}\ot E_m\in P_m$
by \cite[Lemma 3.10]{SunChen}. Thus, we have
$$\begin{array}{c}
\l y^tz^m=[V_{ms+t+1}(\l)]+\sum\limits_{1\<i\<[\frac{t}{2}]}\frac{t-2i+1}{t-i+1}\binom{t}{i}[V_{t+1-2i+ms}(\chi^i\l)]+\l[E].
\end{array}$$
It follows that $\{\l y^tz^m|0\<t\<s-1, m\>0, \l\in\hat{G}\}$ is linearly independent over $\mathbb{Z}$.
\end{proof}

Let $\mathbb{Z}{\hat G}[y,z]$ be the polynomial algebra over the group ring ${\mathbb Z}{\hat G}$ in two variables $y$ and $z$,
and let $I$ be the ideal of $\mathbb{Z}{\hat G}[y,z]$ generated by the element
$$\begin{array}{l}y^{s}-(1+\chi)\sum_{i=0}^{[\frac{s-1}{2}]}(-1)^i\binom{s-1-i}{i}\chi^iy^{s-1-2i}\\
\hspace{0.35cm}-\sum_{1\<i\<[\frac{s-1}{2}]}\sum_{j=0}^{[\frac{s-2i-1}{2}]}(-1)^j\frac{s-2i}{s-i}\binom{s-1}{i}\binom{s-2i-1-j}{j}\chi^{i+j}y^{s-2i-2j}.
\end{array}$$

\begin{corollary}\colabel{4.11}
$r(\mathcal W)$ is isomorphic to the quotient ring $\mathbb{Z}{\hat G}[y, z]/I$.
\end{corollary}
\begin{proof}
Since $r(\mathcal W)$ is commutative, it follows from \leref{4.1} and \prref{4.6}
that the ring embedding ${\mathbb Z}{\hat G}\hookrightarrow r(\mathcal W)$ can be extended into
a ring epimorphism $\phi: \mathbb{Z}{\hat G}[y,z]\ra r(\mathcal W)$ such that
$\phi(y)=[V_2(\e)]$ and $\phi(z)=[V_{s+1}(\e)]$.
By \leref{4.8}, $\phi(I)=0$. Hence $\phi$ induces a ring epimorphism $\ol{\phi}: \mathbb{Z}{\hat G}[y, z]/I\ra r(\mathcal W)$ such that
$\phi=\ol{\phi}\circ\pi$, where $\pi: \mathbb{Z}{\hat G}[y,z]\ra \mathbb{Z}{\hat G}[y,z]/I$ is the canonical epimorphism.
Let $\ol{u}=\pi(u)$ for any $u\in\mathbb{Z}{\hat G}[y,z]$. Then by the definition of $I$,
one can see that $\mathbb{Z}{\hat G}[y,z]/I$ is generated, as a $\mathbb{Z}$-module,
by $\{\ol{\l}\ol{y}^t\ol{z}^m|0\<t\<s-1,m\>0,\l\in\hat{G}\}$.
By \prref{4.10}, $\{\ol{\phi}(\ol{\l}\ol{y}^t\ol{z}^m)|0\<t\<s-1,m\>0,\l\in\hat{G}\}$ is a $\mathbb Z$-basis of $r(\mathcal W)$.
This implies that $\ol{\phi}$ is a ring monomorphism, and so it is a ring isomorphism.
\end{proof}

\subsection{The case of $|\chi|<\infty$}\selabel{4.3}
In this subsection, we consider the case of $|\chi|<\infty$.
Throughout this subsection, assume that $|\chi|=\ol{s}<\infty$.
Let $|q|=s$. Then $s>1$ and $\ol{s}=s's$ for some integer $s'\>1$.
Let $y=[V_2(\e)]$, $z=[V_{s+1}(\e)]$ and $x_{\b}=[V_1(\e,\b)]$ in $r(\mathcal W)$, where $\b\in k^{\times}$.

We first notice that when $s'=1$ (i.e. $\ol{s}=s$), Lemmas 3.1-3.3 and Proposition 3.4 coincide with
\cite[Lemmas 3.12-3.14 and Theorem 3.15]{SunChen} respectively, \prref{3.12} and
\thref{3.14} coincide with \cite[Theorem 3.6 and Theorem 3.7]{SunChen} respectively, and \leref{3.13}
coincides with \cite[Proposition 3.18]{WangYouChen}.

Let $R'$ be the $\mathbb Z$-submodule of $r(\mathcal W)$ generated by $\{[V_t(\l)]|\l\in\hat{G}, t\>1\}$.
Then by \prref{3.8}, one knows that $R'$ is a subring of $r(\mathcal W)$.
From the notice above, one can see that if we replace $r(\mathcal W)$ by $R'$ in the last subsection
then all the statements there are still valid. Hence $R'$ is generated, as a ring, by $\hat{G}\cup\{y,z\}$,
and $R'$ has a $\mathbb{Z}$-basis $\{\l y^tz^m|0\<t\<s-1, m\>0, \l\in\hat{G}\}$.
Moreover, $R'$ is isomorphic to the quotient ring $\mathbb{Z}\hat{G}[y,z]/I$, where the polynomial ring
$\mathbb{Z}\hat{G}[y,z]$ and its ideal $I$ are defined as in the last subsection.

\begin{lemma}\lelabel{4.12}
$r(\mathcal W)$ is generated as a ring by $\hat{G}\cup\{y, z, x_{\b}|\b\in k^{\times}\}$.
\end{lemma}
\begin{proof}
Let $R''$ be the subring of $r(\mathcal W)$ generated by $\hat{G}\cup\{y, z, x_{\b}|\b\in k^{\times}\}$.
Then $R'\subset R''$. By \prref{2.1} and the discussion above, it is enough to show that $[V_t(\s,\b)]\in R''$
for all $t\>1$, $\b\in k^{\times}$ and $[\s]\in\hat{G}/\langle\chi\rangle$. We proof it by induction on $t$.
By \prref{3.12}, one gets that $V_1(\s)\ot V_1(\e, \b)\cong V_1(\s,\b)$ and
$V_{s+1}(\e)\ot V_1(\s,\b)\cong (s-1)V_1(\s,\b)\oplus V_2(\s,\b)$
for any $\b\in k^{\times}$ and $[\s]\in\hat{G}/\langle\chi\rangle$,
and hence $[V_1(\s,\b)]=\s x_{\b}\in R''$ and $[V_2(\s,\b)]=(z-s+1)[V_1(\s,\b)]\in R''$.
Now let $t\>2$ and assume that $[V_m(\s,\b)]\in R''$ for all $1\<m\<t$, $[\s]\in\hat{G}/\langle\chi\rangle$ and $\b \in k^{\times}$.
Again by \prref{3.12}, we have
$V_{s+1}(\epsilon)\ot V_t(\s,\b)\cong (s-1)V_t(\s,\b)\oplus V_{t-1}(\s,\b)\oplus V_{t+1}(\s,\b)$,
and so $[V_{t+1}(\s,\b)]=(z-s+1)[V_t(\s,\b)]-[V_{t-1}(\s,\b)]\in R''$.
This completes the proof.
\end{proof}

For $m\>1$, let $Q_m$ be the set of $H$-modules defined by
$$Q_m=\{\oplus_{i=1}^n l_iV_{t_i}(\s_i,\b_i)|l_i\>0, 1\<t_i\<m, [\s_i]\in\hat{G}/\langle\chi\rangle,\b_i\in k^{\times}, 1\<i\<n\}.$$

\begin{proposition}\prlabel{4.13}
The following set is a $\mathbb{Z}$- basis of $r(\mathcal W)$:
$$\{\l y^tz^m, \s z^mx_{\b}|0\<t\<s-1,\l\in \hat G, [\s]\in\hat{G}/\langle\chi\rangle, m\>0,\b\in k^{\times},\}$$
\end{proposition}
\begin{proof}
Let $N$ be the $\mathbb{Z}$-submodules of $r(\mathcal W)$ generated by $\{[V_t(\s,\b)]|t\>1, [\s]\in\hat{G}/\langle\chi\rangle,\b \in k^{\times}\}$.
Then $r(\mathcal W)=R'\oplus N$ as $\mathbb{Z}$-modules. By \prref{4.10}, it is enough to show that
$\{\s z^mx_{\b}|[\s]\in\hat{G}/\langle\chi\rangle, m\>0,\b\in k^{\times},\}$ is a
$\mathbb{Z}$-basis of $N$. For $m\>1$, by \leref{4.9} and \prref{3.12}, we have
$V_{s+1}(\epsilon)^{\ot m}\ot V_1(\epsilon,\b)
\cong V_{m+1}(\epsilon,\b)\oplus(s-1)V_m(\epsilon,\b)\oplus E_m\ot V_1(\epsilon,\b)$,
where $E_m\in P_m$, $P_m$ is given as in the last subsection. Again by \prref{3.12},
one can see that $(s-1)V_m(\epsilon,\b)\oplus E_m\ot V_1(\epsilon,\b)\in Q_m$, and
hence $\s z^m x_{\b}=[V_{m+1}(\s,\b)]+[H]\in N$ for some $H\in Q_m$, where $[\s]\in\hat{G}/\langle\chi\rangle$, $m\>1$ and $\b\in k^{\times}$.
By the proof of \leref{4.12}, $\s x_{\b}=[V_1(\s,\b)]\in N$.
It follows that $\{\s z^mx_{\b}|[\s]\in\hat{G}/\langle\chi\rangle, m\>0,\b\in k^{\times}\}$ is linearly independent over
$\mathbb{Z}$ and that $N$ is generated, as a $\mathbb Z$-module, by $\{\s z^mx_{\b}|[\s]\in\hat{G}/\langle\chi\rangle, m\>0,\b\in k^{\times}\}$.
This completes the proof.
\end{proof}

Since $k$ is an algebraically closed field of characteristic zero, the map $(\hspace{0.2cm})^{s'}: k^{\times}\ra k^{\times}$,
$\a\mapsto\a^{s'}$ is a (multiplicative) group epimorphism with the kernel $\langle\xi\rangle$, the subgroup of
$k^{\times}$ generated by $\xi$, where $\xi\in k$ is a root of unity of order $s'$.
Let $[\a]$ denote the image of $\a$ under the canonical epimorphism $k^{\times}\ra k^{\times}/\langle\xi\rangle$.
Then there is a group isomorphism $k^{\times}/\langle\xi\rangle\ra k^{\times}$, $[\a]\mapsto \a^{s'}$.
Identifying the group $k^{\times}$ with $k^{\times}/\langle\xi\rangle$ via the isomorphism,
we may denote $x_{\a^{s'}}$ by $x_{[\a]}$ for any $\a\in k^{\times}$. When $\ol{s}=s$ (or $s'=1$ equivalently),
$x_{[\a]}=x_{\a}$ for any $\a\in k^{\times}$.  By \leref{4.12} and \prref{4.13},
one gets the following corollary.

\begin{corollary}\colabel{4.14}
$(1)$ $r(\mathcal W)$ is generated, as a ring, by $\hat{G}\cup\{y, z, x_{[\a]}|[\a]\in k^{\times}/\langle\xi\rangle\}$.\\
$(2)$ $r(\mathcal W)$ has a $\mathbb{Z}$-basis as follows:
$$\{\l y^tz^m, \s z^mx_{[\a]}|0\<t\<s-1, m\>0, \l\in\hat{G}, [\s]\in\hat{G}/\langle\chi\rangle, [\a]\in k^{\times}/\langle\xi\rangle\}.$$
\end{corollary}

\begin{lemma}\lelabel{4.15}
Let $\l\in\hat{G}$ and $[\a], [\b]\in k^{\times}/\langle\xi\rangle$. Then the following equations hold in $r(\mathcal W)$.\\
$(1)$ $\chi x_{[\a]}=x_{[\a]}\chi=x_{[\a]}$.\\
$(2)$ $x_{[\a]}\l=\l x_{[\l(a)^s\a]}$.\\
$(3)$ $yx_{[\a]}=x_{[\a]}y=2x_{[\a]}$.\\
$(4)$ $zx_{[\a]}=x_{[\a]}z$.\\
$(5)$ $x_{[\a]}x_{[-\a]}=x_{[-\a]}x_{[\a]}=\sum_{1\<i\<s'-1}sx_{[(1-\xi^i)\a]}
+\sum_{i=0}^{\ol{s}-1}\sum_{j=0}^{[\frac{s-1}{2}]}(-1)^j\binom{s-1-j}{j}\chi^{i}y^{s-1-2j}$.\\
$(6)$ If $[\b]\neq[-\a]$, then $x_{[\a]}x_{[\b]}=x_{[\b]}x_{[\a]}=\sum_{i=0}^{s'-1}sx_{[\a+\b\xi^{i}]}$.
\end{lemma}
\begin{proof}
Since $V_1(\chi\s,\a)\cong V_1(\s,\a)$, (1)-(4) follow from \prref{3.12},
(5)-(6) follow from \leref{3.13}, \leref{4.7}(1) and \cite[Lemma 3.8]{SunChen}
since $\sum_{i=0}^{\ol{s}-1}\chi^{i+j}=\sum_{i=0}^{\ol{s}-1}\chi^{i}$ for any $j\in\mathbb Z$.
\end{proof}

Let $X=\{y,z,x_{[\a]}|[\a]\in k^{\times}/\langle\xi\rangle\}$ and $\mathbb{Z}[X]$ be the corresponding
polynomial ring. Then there is a right action of $\hat{G}$ on $\mathbb{Z}[X]$, which can be described as follows.
For any $\l\in\hat{G}$, one can define a map $(\hspace{0.15cm})^{\l}: X\ra X$ by $y^{\l}=y$, $z^{\l}=z$
and $x_{[\a]}^{\l}=x_{[\l(a)^s\a]}$ for any $[\a]\in k^{\times}/\langle\xi\rangle$.
Obviously, $(\hspace{0.15cm})^{\l}$ is a bijection. Thus, $(\hspace{0.15cm})^{\l}$ induces a ring automorphism of $\mathbb{Z}[X]$,
denoted still by $(\hspace{0.15cm})^{\l}: \mathbb{Z}[X]\ra \mathbb{Z}[X]$, $r\mapsto r^{\l}$.
It is easy to see that $(r^{\l})^{\s}=r^{\l\s}$ and $r^{\e}=r$ for all $r\in\mathbb{Z}[X]$ and $\l, \s\in\hat{G}$.
Hence the group $\hat G$ acts on the ring $\mathbb{Z}[X]$ from the right. Moreover,
$(\hspace{0.15cm})^{\chi}$ is the identity map on $\mathbb{Z}[X]$ by $\chi(a)^s=1$.

Now one can form a skew group ring $\mathbb{Z}[X]\sharp\hat{G}$ (see \cite[p.22]{McRo} for the definition of a skew group ring).
$\mathbb{Z}[X]\sharp\hat{G}$ is a free right $\mathbb{Z}[X]$-module with $\hat{G}$ as a basis, and with the multiplication
defined by $(\l r)(\s t)=(\l\s)(r^{\s}t)$ for $\l, \s\in\hat{G}$, $r, t\in\mathbb{Z}[X]$.
Obviously, $\mathbb{Z}[X]\sharp\hat{G}$ contains $\hat G$ as a subgroup of its group of units, and $\mathbb{Z}[X]$ as a subring,
under the identifications $\l=\l 1$, $\l\in\hat{G}$, and $r=\e r$, $r\in \mathbb{Z}[X]$, respectively.
Note that each element of $(\mathbb{Z}[X]\sharp\hat{G})$ has a unique expression as
$\sum_{\l\in\hat{G}}\l r_{\l}$ with $r_{\l}\in \mathbb{Z}[X]$ being almost all zero.

Let $J$ be the ideal of $\mathbb{Z}[X]\sharp\hat{G}$ generated by the following subset
$$U=\left\{\left.\begin{array}{l}
y^{s}-(1+\chi)\sum_{i=0}^{[\frac{s-1}{2}]}(-1)^i\binom{s-1-i}{i}\chi^iy^{s-1-2i}\\
-\sum_{1\<i\<[\frac{s-1}{2}]}\sum_{j=0}^{[\frac{s-2i-1}{2}]}(-1)^j\frac{s-2i}{s-i}\binom{s-1}{i}\binom{s-2i-1-j}{j}\chi^{i+j}y^{s-2i-2j},\\
\chi x_{[\a]}-x_{[\a]}, yx_{[\a]}-2x_{[\a]}, x_{[\a]}x_{[\b]}-\sum_{i=0}^{s'-1}sx_{[\a+\b\xi^{i}]},\\
x_{[\a]}x_{[-\a]}-\sum_{1\<i\<s'-1}sx_{[(1-\xi^i)\a]}\\
-\sum_{i=0}^{\ol{s}-1}\sum_{j=0}^{[\frac{s-1}{2}]}(-1)^j\binom{s-1-j}{j}\chi^{i}y^{s-1-2j}\\
\end{array}
\right|
\begin{array}{l}
[\a], [\b]\in k^{\times}/\langle\chi\rangle\\
\text{with }[\b]\neq[-\a]\\
\end{array}
\right\}.$$

\begin{theorem}\label{4.4}
The Green ring $r(\mathcal W)$ is isomorphic to the quotient ring $(\mathbb{Z}[X]\sharp\hat{G})/J$.
\end{theorem}
\begin{proof}
By \leref{4.15}(3)-(6) and the paragraph before \leref{4.12}, one knows that for any $\a, \b\in k^{\times}$, the elements
$[V_2(\e)]$, $[V_{s+1}(\e)]$, $[V_1(\e, \a^{s'})]$ and $[V_1(\e, \b^{s'})]$ are pairwise commutative in $r(\mathcal W)$.
Hence there is a ring homomorphism $f$ from $\mathbb{Z}[X]$ to $r(\mathcal W)$ such that
$f(y)=[V_2(\e)]$, $f(z)=[V_{s+1}(\e)]$ and $f(x_{[\a]})=[V_1(\e, \a^{s'})]$ for all $[\a]\in k^{\times}/\langle\chi\rangle$.
Define a map $\phi: \mathbb{Z}[X]\sharp\hat{G}\ra r(\mathcal W)$ by
$\phi(\sum_{\l\in\hat{G}}\l r_{\l})=\sum_{\l\in\hat{G}}\l f(r_{\l})$.
Obviously, $\phi$ is a $\mathbb Z$-module map.
By \leref{4.15}(2) and the paragraph before \leref{4.12},  we have $f(y)\l=\l f(y)=\l f(y^{\l})$, $f(z)\l=\l f(z)=\l f(z^{\l})$
and $f(x_{[\a]})\l=\l f(x_{[\l(a)^s\a]})=\l f(x_{[\a]}^{\l})$
for all $\l\in\hat{G}$ and $[\a]\in k^{\times}/\langle\chi\rangle$.
Hence $f(r)\l=\l f(r^{\l})$ for all $\l\in\hat{G}$ and $r\in\mathbb{Z}[X]$.
Thus, we have $\phi(\l r)\phi(\s t)=\l f(r)\s f(t)=\l\s f(r^{\s})f(t)=\l\s f(r^{\s}t)=\phi((\l\s)(r^{\s}t))=\phi((\l r)(\s t))$
for all $\l, \s\in\hat{G}$ and $r, t\in\mathbb{Z}[X]$, and so $\phi$ is a ring homomorphism.
Then by \coref{4.14}(1), one can see that $\phi$ is a ring epimorphism.
By \leref{4.8} and \leref{4.15}(1), (3), (5)-(6),
one can check that $\phi(u)=0$ for any $u\in U$, and so $\phi(J)=0$.
Thus, $\phi$ induces a ring epimorphism $\ol{\phi}: (\mathbb{Z}[X]\sharp\hat{G})/J\ra r(\mathcal W)$
such that $\ol{\phi}(\ol{u})=\phi(u)$ for all $u\in\mathbb{Z}[X]\sharp\hat{G}$,
where $\ol{u}$ is the image of $u$ under the canonical epimorphism $\mathbb{Z}[X]\sharp\hat{G}\ra(\mathbb{Z}[X]\sharp\hat{G})/J$.
Note that $\{\l y^tz^mx_{[\a_1]}\cdots x_{[\a_n]}|\l\in\hat{G}, t, m, n\>0, [\a_1],\cdots, [\a_n]\in k^{\times}/\langle\chi\rangle\}$
is a $\mathbb Z$-basis of $\mathbb{Z}[X]\sharp\hat{G}$, here we regard $x_{[\a_1]}\cdots x_{[\a_n]}=1$ when $n=0$.
It follows from the definition of $J$ that $(\mathbb{Z}[X]\sharp\hat{G})/J$ is generated, as a $\mathbb{Z}$-module, by
$$\{\ol{\l}\ol{y}^t\ol{z}^m, \ol{\s} \ol{z}^m\ol{x_{[\a]}}|0\<t\<s-1, m\>0, \l\in\hat{G}, [\s]\in\hat{G}/\langle\chi\rangle, [\a]\in k^{\times}/\langle\chi\rangle\}.$$
By \coref{4.14}(2), the image of the above set under $\ol{\phi}$ is a $\mathbb Z$-basis of $r(\mathcal W)$.
This implies that $\ol{\phi}$ is an injection, and so it is a ring isomorphism.
\end{proof}

{\bf Acknowledgments}
This work is supported by National Natural Science Foundation of China (Grant No. 11571298, 11711530703).

\end{document}